\documentclass[12pt]{amsart}
\usepackage{amsthm}
\usepackage[all]{xy}
\usepackage{amsmath,amsfonts,textcomp,longtable,mathrsfs}
\usepackage{amssymb}
\usepackage{amscd}

\newtheorem{thm}{Theorem }[section]

\newtheorem{lemma}[thm]{Lemma }

\newtheorem{prop}[thm]{Proposition }

\newtheorem{corollary}[thm]{Corollary }

\theoremstyle{definition}
\newtheorem{deff}[thm]{Definition }

\newtheorem{rem}[thm]{Remark }

\newtheorem{ex}[thm]{Example }

\newtheorem{notation}[thm]{Notation }

\def\RR{{\mathbb R}}

\def\CC{{\mathbb C}}

\def\FF{{\mathbb F}}
\def\FFF{\overline{\mathbb F}}

\def\HH{{\mathbb H}}

\def\NN{{\mathbb N}}

\def\QQ{{\mathbb Q}}
\def\BQQ{{\overline{\mathbb Q}}}
\def\ZZ{{\mathbb Z}}

\def \bra#1\ket {\mathop{\vphantom{#1}\left<\smash{#1}\right>}\nolimits}
\def\Die{Dieudonn\'e }

\def\ord{{\mathrm{ord}}}
\def\op{\mathrm{op}}
\def\max{{\mathrm{max}}}

\DeclareMathOperator{\inv}{inv}

\DeclareMathOperator{\Hom}{Hom}
\DeclareMathOperator{\End}{End}

\DeclareMathOperator{\coker}{Coker}

\DeclareMathOperator{\Gal}{Gal}
\DeclareMathOperator{\Spec}{Spec}

 \DeclareMathOperator{\Br}{Br}
 
\DeclareMathOperator{\AV}{AV} 
\DeclareMathOperator{\Mat}{Mat}

\def\rig{\mathrm{rig}}

\def\Md#1{\operatorname{{#1}--\mathrm{mod}}}
\def\TF#1{\operatorname{{#1}--\mathrm{mod}^\mathrm{{TF}}}}
\def\RM#1{\operatorname{{#1}--\mathrm{mod}^\rig}}

\DeclareMathOperator{\Del}{Del}
\DeclareMathOperator{\TDel}{TDel}
\DeclareMathOperator{\DieW}{Die}

\DeclareMathOperator{\SSDel}{SSDel}
\DeclareMathOperator{\GDel}{GDel}

\newcommand{\leg}[2]{\genfrac{(}{)}{}{}{#1}{#2}}

\def\C{\mathcal{C}}
\def\cC{\mathcal{C}}
\def\D{\mathcal{D}}

\def\L{\mathcal{L}}

\def\M{\mathcal{M}}
\def\hM{\widehat{\mathcal{M}}}

\def\F{\mathcal{F}}

\def\OO{\mathcal{O}}
\def\cT{\mathcal{T}}
\def\hT{\widehat{\mathcal{T}}}

\def\p{\mathfrak{p}}
\def\q{\mathfrak{q}}
\def\T{\Lambda}

\def\CG{\overline{G}}

\def\plim{\mathop{{\lim\limits_{\longleftarrow}}}\nolimits}
\def\DB{\overline{\Delta}}

\def\locp{\left[\frac{1}{p}\right]}

\newcommand \eps {\varepsilon}
\renewcommand \phi {\varphi}

\begin{document}
\author{Sergey Rybakov}
\thanks{Supported by the Israel Science Foundation,  grant No.  1405/22}
\address{Department of Mathematics, Ben-Gurion University of the Negev, Israel}
\email{rybakov.sergey@gmail.com}%
\title[Twisted Deligne modules]
{Twisted Deligne modules and \\ abelian varieties over finite fields}
\date{}


\begin{abstract}
Isogeny classes of abelian varieties over finite fields were described by Tate and Honda. Deligne proved that the category of ordinary abelian varieties over a finite field is equivalent to the category of ordinary Deligne modules. 
Centeleghe and Stix extended Deligne's results to the whole category of abelian varieties,
but, if the base finite field is not prime, then the target category is not related to Deligne modules. In this paper, we introduce generalized and twisted Deligne modules and give a more direct generalization of the Deligne theorem.
\end{abstract}

\dedicatory{Dedicated to the memory of Yuri Zarhin.}

\maketitle

\section{Introduction.}
\subsection{Abelian varieties and Deligne modules.}
Over complex numbers there is an equivalence of categories from the category of complex polarized abelian varieties to the category of polarized Hodge structures of type $(1,1)$:
\[A\mapsto H_1(A,\ZZ).\]
Moreover, the Hodge structure on the group of singular cohomology $H_1(A,\ZZ)$ is uniquely determined by $A$ up to a unique isomorphism. 

If the base field is finite, the situation is more complicated. 
According to the results of Tate and Honda, a simple abelian variety $A$ over a finite field is determined up to isogeny by an eigenvalue of the Frobenius endomorphism on the first étale cohomology group of $A$, the Weil number. More generally, an abelian variety is determined up to isogeny by its Weil polynomial. There are many results that describe the category of abelian varieties over finite fields under various 
assumptions~\cite{Del, CS15, CS2, EC18, BKM, OO}. 
Let us briefly recall some of the results related to Deligne modules.

Let $\FF_q$ be a finite field.

\begin{deff}
 \emph{A Deligne module} over $\FF_q$ is a free and finitely generated abelian group endowed with an action of two operators $F$ and $V$ such that $FV=VF=q$, and $F\otimes\QQ$ is semi-simple.
 A Deligne module is \emph{ordinary} if the eigenvalues of $F$ are ordinary Weil numbers.
\end{deff}

\noindent

\begin{thm}~\cite{Del}\label{Del_thm}
The category of ordinary abelian varieties over $\FF_q$ is equivalent to the category of ordinary Deligne modules. If $D$ is a Deligne module of $A$, then the Weil polynomial $f_A$ and the Tate module $T_\ell(A)$ of $A$ are related to $D$ as follows:
\[f_A(t)=\det(t-F|D),\text{ and }\; T_\ell(A)\cong D\otimes_\ZZ\ZZ_\ell.\] 
\end{thm} 

The proof relies on the fact that an ordinary abelian variety A has a \emph{canonical lift} $A_W$ to the ring of Witt vectors $W(\FF_q)$. 
Choose a (non-continuous) homomorphism $W(\FF_q)\to\CC$. The functor sends $A$ to the singular homology group of the lift $A_W\otimes\CC$:
\[A\mapsto H_1(A_W\otimes\CC,\ZZ).\]
The functor depends on the choice of the homomorphism $W(\FF_q)\to\CC$ and, therefore, is \emph{not} unique.

Centeleghe and Stix extended this result to abelian varieties over a prime finite field as follows. 
Let $A$ be an abelian variety over $\FF_q$ with the Weil polynomial $f_A$. 
The set of roots of $f_A$ is called the \emph{Weil support of $A$}. 
\emph{The Weil support of a Deligne module $T$} is the set of eigenvalues of $F\otimes\QQ$. 
 Let $\Pi$ be a $\Gal{\Bar\QQ/\QQ}$-invariant set of Weil numbers over $\FF_q$. Define $\AV_\Pi$ as the category of abelian varieties over $\FF_q$ with the Weil support in $\Pi$.

\begin{thm}\label{CS_thm}\cite{CS15}
Assume that $q=p$ and $\pm\sqrt{p}\notin\Pi$.
The category of abelian varieties $AV_\Pi$ over $\FF_p$ is anti-equivalent to the category of Deligne modules over $\FF_p$ with Weil support in $\Pi$.
If $D$ is a Deligne module of $A$, then
\[f_A(t)=\det(t-F|D),\text{ and }\; T_\ell(A)\cong D\otimes_\ZZ\ZZ_\ell.\] 
%\[f_A(t)=\det(t-F|D).\] 
\end{thm} 

\begin{rem}
The endomorphism algebra of a simple abelian variety $A$ over a finite field can be non-commutative. In this case, there is no natural way to assign a Deligne module to the abelian variety. Indeed, assume that there exists a functorial Deligne module $D$ for $A$, where $\End^\circ(A)$ is simple and not commutative, that is, the dimension of $\End^\circ{(A)}$ is equal to $e\deg(f_A)$ for some $e>1$. 
The algebra ${\End(A)}^\op$ acts on $D$; therefore, the simple algebra
 $\End^\circ{(A)^\op}$ acts on the $\QQ$-vector space $D\otimes_\ZZ\QQ$ of dimension $\deg(f_A)$.
 A contradiction.
\end{rem}

Centeleghe and Stix proved a more general result on the category of all abelian varieties.
They construct an equivalence from the category of abelian varieties with Weil support in $\Pi$ to a category of modules over the endomorphism algebra of a balanced abelian variety~\cite{CS2}. 
Over a prime field, the target category is equivalent to the category of Deligne modules, but, in general, this category is rather inexplicit.

\subsection{Generalized Deligne modules}
We fix once and for all a natural number $n$ and a finite field $\FF_q$ of characteristic $p$, where $q=p^n$. We extend Deligne's theorem to the category $\AV_\Pi$ of abelian varieties with Weil support in $\Pi$, where $\Pi$ is a Galois invariant set of Weil numbers over $\FF_q$.

Let $L/\QQ$ be a cyclic extension of degree $n$ such that $p$ is inert; for the ring of integers of $L$ we have
\[\OO_L\otimes_\ZZ\ZZ_p\cong W(\FF_q).\]
There is a generator $g$ of the Galois group $\Gal(L/\QQ)$ corresponding to the Frobenius automorphism at $p$. The following definition is a counterpart to the definition of a covariant \Die module over $\OO_L$.

\begin{deff}
{A generalized Deligne module} $D$ is a finitely generated locally free $\OO_L$-module with a pair of semi-linear operators $G$ and $\CG$ such that $G^n$ is semi-simple and 
for any $x\in \OO_L$ and $y\in D$ we have
\[G(xy)=g^{-1}(x)G(y),\quad \CG(xy)=g(x)\CG(y), \text{ and } G\CG=\CG G=p.\]  
The set of eigenvalues of the $\OO_L$-linear map $G^n$ is called \emph{the Weil support of $D$}. 
\end{deff}

Let $\Pi$ be a Galois invariant set of Weil numbers. Choose a subset of representatives
$\Pi_0\subset \Pi$, that is, for any Galois orbit in $\Pi$ there is a unique representative in $\Pi_0$. Assume that $L$ and $\QQ(\Pi)$ are linearly disjoint over $\QQ$:
\[L(\Pi)\cong L\otimes_\QQ\QQ(\Pi).\]
 There is a cyclic algebra $\HH_{\Pi,L}$ such that a generalized Deligne module over $\OO_L$ is a module over some order in the opposite algebra $\HH_{\Pi,L}^\op$.
Namely $\HH_{\Pi,L}=\oplus_{\pi\in\Pi_0}\HH_\pi$, where $\HH_\pi$ is generated over $L\otimes\QQ(\pi)$ by a formal variable $G$ with relations \[G(x\otimes y)=g(x)\otimes yG,\text{ and }\; G^n=\pi,\]
where $x\in L$, and $y\in \QQ(\pi)$.

We would like to have an equivalence from $\AV_\Pi$ to the category of generalized Deligne modules over $\OO_L$ with Weil support in $\Pi$. The algebra $\HH_\pi$ is the first obstruction to the existence of such an equivalence. Indeed, if $D$ is a generalized Deligne module of $A$, then the endomorphism algebra of $D$ is isomorphic to $\End^\circ(A)$; therefore, for any $\pi\in\Pi$, the algebra
$\HH_\pi$ is Brauer equivalent to $\End^\circ(B_\pi)$, where $B_\pi$ is a simple abelian variety with Weil number $\pi$ (see Section~\ref{TH_theory} for more details).
Since the cocycle that defines $\HH_\pi$ is fixed, we have to find a field $L$ such that this obstruction disappears. 

The second obstruction is more delicate. We now formulate it and discuss it in more detail in Section~\ref{twisted_subsec}. Let $F_A, V_A\in\End(A)$ be the Frobenius and Verschiebung endomorphisms of an abelian variety $A$. There is a ring $R_\Pi=\ZZ[F_\Pi,V_\Pi]$ such that
for any abelian variety $A$ with Weil support equal to $\Pi$ we have a monomorphism
$R_\Pi\to \End(A)$, where $F_\Pi\mapsto F_A$ and $V_\Pi\mapsto V_A$ 
(see Section~\ref{TH_theory}). 
Let $\L_\Pi$ be the set of primes $\ell\neq p$ such that
$R_\Pi\otimes_\ZZ\ZZ_\ell$ is not maximal in the semi-simple algebra $R_\Pi\otimes_{\ZZ}\QQ_\ell$.
We would like to find $L$ such that any $\ell\in\L_\Pi$ splits completely in $L/\QQ$.
Let us precisely formulate the desired properties for $L$.

\begin{deff}\label{strict_adm_def}
A cyclic extension $L/\QQ$ of degree $n$ is called \emph{$\Pi$-admissible} if
\begin{enumerate}
\item $p$ is inert in $L/\QQ$;
\item if $n>1$, then $L/\QQ$ is tamely ramified at a single prime $s\not\in\L_\Pi$;
\item any $\ell\in\L_\Pi$ splits completely in $L/\QQ$;
\item for any $\pi\in\Pi$ the algebra $\HH_\pi$ is Brauer equivalent to $\End^\circ(B_\pi)$.
  \end{enumerate}
\end{deff}

For some sets of Weil numbers $\Pi$ there are no $\Pi$-admissible extensions (see Example~\ref{ex_main}), but, according to Theorem~\ref{Pi_adm_ext2}, in general situations, such extensions exist. 
We can now state an equivalence theorem for generalized Deligne modules.

\begin{thm}\label{main_thm_Del}
Assume that there exists a $\Pi$-admissible extension $L/\QQ$ ramified at a prime $s$.
Then there is an equivalence of categories
\[\D_\Pi:\AV_\Pi[\frac{1}{s}]\to\GDel_{\Pi,L}[\frac{1}{s}],\] where $\GDel_{\Pi,L}$ is the category of generalized Deligne modules over $\OO_L$ with Weil support in $\Pi$.
Moreover, if $A$ is an abelian variety with Weil support in $\Pi$, and 
$\ell\not\in\{s,p\}$, then there is a (non-canonical) isomorphism of 
$R_\Pi\otimes_\ZZ\OO_L\otimes_\ZZ\ZZ_\ell$-modules:
\[\D_\Pi(A)\otimes_\ZZ\ZZ_\ell\cong T_\ell(A)\otimes_\ZZ\OO_L.\]
There is a natural structure of a covariant \Die module on $\D_\Pi(A)\otimes_\ZZ\ZZ_p$ such that \[\D_\Pi(A)\otimes_\ZZ\ZZ_p\cong M_*(A),\]
where $M_*(A)$ is the covariant \Die module of $A$.
\end{thm}

\subsection{A general equivalence theorem.}\label{general_subsec}
Theorem~\ref{main_thm_Del} follows from a more general result that we prove in Section~\ref{general_sec}. Here, we formulate it in a simplified form.

Fix a cyclic extension $L/\QQ$ of degree $m$ such that $p$ is unramified.
Put \[n_0=m/[\Hat{L}_p:\QQ_p],\]
where $\Hat{L}_p$ is the completion of $L$ at a prime ideal over $p$.
Choose a generator $g$ of $\Gal(L/\QQ)$ such that $g^{n_0}$ is the Frobenius automorphism at $p$. 
Let \[\HH_{\Pi,L}(\lambda)=\oplus_{\pi\in\Pi_0}\HH_{\pi,L}(\lambda),\] 
where $\HH_{\pi,L}(\lambda)$ is the cyclic algebra $(L(\pi)/\QQ(\pi),g,\lambda)$, and
$\lambda\in R_\Pi$.  

Assume that there exists an abelian variety $A_{\Pi,L}$ such that 
\[\HH_{\Pi,L}(\lambda)\cong \End^\circ(A_{\Pi,L}).\]
Fix an order $\T$ in $\HH_{\Pi,L}(\lambda)$ such that $R_\Pi\otimes_\ZZ\OO_L\subset\T$.
Let $\ell\neq p$ be a prime. 
Assume that there exists a $\T^\op$-module structure on $T_{\ell,\Pi,L}=R_\Pi\otimes_\ZZ\OO_L\otimes_\ZZ\ZZ_\ell.$ 
We say that $T_{\ell,\Pi,L}$ is \emph{$\ell$-balanced} if 
\[\T^\op\otimes_{\ZZ}\ZZ_\ell\cong\End_{R_\Pi\otimes\ZZ_\ell}(T_{\ell,\Pi,L}).\]
We say that a collection of $\ell$-balanced modules $T_{\ell,\Pi,L}$ for all $\ell\neq p$
is \emph{a balanced structure on $\T$}.

Let $\TF{\T^\op}$ be the category of left $\T^\op$-modules that are finitely generated and locally free over $\OO_L$. If $A_\T$ is an abelian variety with a homomorphism $\T\to \End(A_\T)$, then there is a covariant functor \[\D_\T:\AV_\Pi\to\TF{\T^\op}\] given by the formula
$A\mapsto \Hom_\ZZ(\Hom(A,A_\T),\ZZ)$.

\begin{thm}\label{main_simple}
    Let $\T\subset \HH_{\Pi,L}(\lambda)$ be an order with a balanced structure. 
    Then there exists an abelian variety $A_\T$ such that
    \[\T\locp\cong\End(A_\T)\locp\] and the functor $A\mapsto (\D_\T(A),M_*(A))$ is an equivalence from        $\AV_\Pi$ to the product category
        \[\TF{\T^\op\locp}\times_{\DieW_\Pi\locp}\DieW_\Pi\]
        of the localization of $\TF{\T^\op}$ and the category $\DieW_\Pi$ of covariant \Die modules with Weil support in $\Pi$ (see Section~\ref{product_sec}).
\end{thm}

Note that if $L=\QQ$ and $\T=R_\Pi$, then the category $\GDel_{\Pi,\QQ}$ of generalized Deligne modules is the category $\Del_\Pi$ of Deligne modules with Weil support in $\Pi$.
We obtain a generalization of Theorem~\ref{Del_thm} and Theorem~\ref{CS_thm}. 
It is also related to~\cite[Theorem 5.2]{BKM}.

\begin{corollary}\label{main_thm_comm}
Assume that for any $\pi\in\Pi$ the endomorphism algebra $\End^\circ(B_\pi)$ is commutative.
Let $L=\QQ$, and let $\T=R_\Pi$.
Then the functor $A\mapsto (\D_\T(A),M_*(A))$ is an equivalence
from $\AV_\Pi$ to the product  \[\Del_\Pi\locp\times_{\DieW_\Pi\locp}\DieW_\Pi.\]
\end{corollary}

\subsection{Twisted Deligne modules.}\label{twisted_subsec}
We introduce a more general notion of $\Omega$-admissible extensions. 
Let $\eps\in\{\pm 1\}$, and let $\Omega=(n,m,m_0,\Pi,\L,\eps)$, where $\L$ is a finite set of prime numbers such that $p\not\in\L$, and $m=nn_0m_0$ for some natural $n_0$ and $m_0$. 

\begin{deff}\label{def_omega}
We say that a cyclic extension $L/\QQ$ is \emph{$\Omega$-admissible} if
\begin{enumerate}
\item $[L:\QQ]=m$;
       \item $L$ is unramified at $p$;
       \item $[\Hat{L}_p:\QQ_p]=nm_0$;
    \item $L$ splits completely at any $\ell\in\L$;
    \item if $m>1$, then $L$ is tamely ramified at a single prime $s\not\in\L$;
%\item $\HH_\pi$ is Brauer equivalent to $\End^\circ(B_\pi)$;
	 \item for any $\pi\in\Pi$ the product $\eps\pi^{m_0}$ is an $m$-th power in $\QQ_s(\pi)$;
\item if $\Pi\cap\RR\neq\emptyset$, then $L$ is a CM-field, and
for any real $\pi\in\Pi$, we have: $\eps\pi^{m_0}<0$.
\end{enumerate}
\end{deff}

\begin{rem}\label{Milne_ref}
  Milne studied a special case of such extensions for Weil $q$-numbers of weight zero in~\cite{MK}. In his definition $n_0=1$ and $\L=\emptyset$.
   Under the assumption of the Generalized Riemann Hypothesis he proved that such extensions exist if $nm_0$ is odd.
\end{rem}

From Corollary~\ref{adm_prime_cor}, and Proposition~\ref{adm_ext} we immediately obtain.
\begin{thm}\label{Omega_main_thm}
    For any Galois invariant set $\Pi$ of Weil numbers over $\FF_q$ and any finite set of primes $\L$ such that $p\not\in\L$ there exist $\eps\in\{\pm 1\}$,
    natural numbers $m_0$ and $n_0$ and infinitely many $\Omega$-admissible extensions $L/\QQ$, where $\Omega=(n,nn_0m_0,m_0,\Pi,\L,\eps)$.
\end{thm}

The field $L$ is the subfield of degree $m$ in $\QQ(\zeta_s)$, where $s$ is a prime that we find using the Chebotarev Density Theorem (see Theorem~\ref{cyclic_ext}).
    
Fix an $\Omega$-admissible extension $L/\QQ$.
Let $\HH_{\Pi,L}=\oplus_{\pi\in\Pi_0}\HH_\pi$, where 
\[\HH_{\pi}=\HH_{\pi,L}(\eps\pi^{m_0}).\]

We show in Theorem~\ref{inv_thm} that for any $\Omega$-admissible extension $L/\QQ$  
there exists an abelian variety $A_{\Pi,L}$ with Weil support in $\Pi$ such that
$\HH_{\Pi,L}\cong\End^\circ(A_{\Pi,L})$. For the proof, according to Honda Theorem~\cite{Ho}, it suffices to show that if $B_\pi$ is a simple abelian variety with Weil number $\pi$, then invariants of the central simple algebras $\HH_\pi$ and $\End^\circ(B_\pi)$ over $\QQ(\pi)$ are equal. Indeed, for a place $v$ of $\QQ(\pi)$ we have:
\begin{itemize}
    \item[$v|p$]: $\inv_v\HH_\pi=\inv_v(\End^\circ(B_\pi))$;
    \item [$v|\ell$]: if $\ell$ is unramified in $L$, then $\inv_v\HH_\pi=\inv_v(\End^\circ(B_\pi))=0$;
    \item[$v|\ell$]: if $\ell$ is ramified in $L$, then 
    $\inv_v\HH_\pi=\inv_v(\End^\circ(B_\pi))=0$ by condition $(6)$ of Definition~\ref{def_omega};
    \item[$\infty$]: if $\pi$ is real, then $\inv_\infty\HH_\pi=\inv_\infty(\End^\circ(B_\pi))$ by condition $(7)$ of Definition~\ref{def_omega}.
\end{itemize}

\begin{rem}
We have to introduce $\Omega$-admissible field extensions with $n_0>1$, because of Example~\ref{ex_main}. 
The numbers $m_0>1$ and $\eps=-1$ appear for the following reason. 
If $\{\pm\sqrt{q}\}\subset\Pi$, then $2$ divides $m_0$ to keep condition $(7)$. 
In turn, $(7)$ is needed, because $\inv_\infty\HH_{\sqrt{q}}$ is non-trivial.
\end{rem}

We discuss twisted Deligne modules in full generality in Section~\ref{sec_tw_Del}. 
Here we present a special case, when $\Pi\cap\RR=\emptyset$ and $m_0=1$.

Assume that $m=n_0n$. Let $L/\QQ$ be an $\Omega$-admissible extension, 
where $\Omega=(n,m,1,\Pi,\L_\Pi,1)$.
By Notation~\ref{twisting_not} there exists a \emph{twisting} $\rho\in \OO_L$ such that 
\[N_{L/\QQ}(\rho)\equiv q\pmod{qp}.\]
Let $\T_{\Pi,L}(\rho)\subset\HH_{\Pi,L}$ be the order generated over $R_\Pi\otimes_\ZZ\OO_L$
by semilinear operators $G$ and $\CG$ such that 
\[G\CG=\rho \text{ and }G^{n_0n}=F_\Pi.\] 

\begin{rem}
The condition $(3)$ of Definition~\ref{strict_adm_def} is needed to prove that for the order 
$\T_{\Pi,L}(\rho)$ there exists a balanced structure. 
Namely, in Proposition~\ref{balanced_order} we prove that if $\ell$ splits completely in $L$, then $T_{\ell,\Pi,L}$ is $\ell$-balanced.  
On the other hand, if $R_\Pi$ is maximal at $\ell$ and $\ell\neq s$ is unramified in $L$, then $T_{\ell,\Pi,L}$ is also $\ell$-balanced. 
\end{rem}

\begin{deff}\label{twisted_module}
{A twisted Deligne module} $D$ is a finitely generated left module over 
$\T_{\Pi,L}(\rho)^\op$ such that $D$ is locally free over $\OO_L$.
We say that $\Pi$ is \emph{the Weil support of $D$}.
\end{deff}

If $n_0=1$ and $\rho=p$, we recover the definition of a generalized Deligne module.
We say that a cyclic extension $L/\QQ$ is \emph{weakly $\Pi$-admissible} if 
$L/\QQ$ is $\Omega$-admissible, where $\Omega=(n,2n,1,\Pi,\L_\Pi,1)$.

\begin{rem}\label{Pell}
    Assume that $n$ is even, and $m=2n$.
    If $L\subset\QQ(\zeta_s)$ is a weakly $\Pi$-admissible extension
    ramified at a prime $s$, then $s\equiv 1\pmod 4$. Hence, $\QQ(\sqrt{s})\subset L$.
    There exists a twisting $\rho\in \QQ(\sqrt{s})$ such that $N_{L/\QQ}(\rho)=q$ if and only if there is an integral solution to the Generalized Pell equation
    \[x^2-sy^2=\pm 4p.\]
\end{rem}

To avoid localization at $s$, in Definition~\ref{H_operator} we introduce an auxiliary endomorphism $H\in\HH_{\Pi,L}$. Denote by $\TDel_{\Pi,L,H}$ the category of twisted Deligne modules over $\OO_L$ with Weil support in $\Pi$ and invariant under $H$.
In Section~\ref{s_section} we prove that the order $\T^\max$ generated by 
$\T_{\Pi,L}(\rho)$ and $H$ is also $s$-balanced. Hence, Theorem~\ref{main_simple} can be applied. In fact, a better result can be proved.

\begin{thm}\label{main_thm_intro}
Assume that $\Pi\cap\RR=\emptyset$.
\begin{enumerate}
\item If $n$ is odd, then there exist infinitely many $\Pi$-admissible extensions.
\item If $n$ is even, then there exist infinitely many weakly $\Pi$-admissible extensions.
\end{enumerate}
In both cases there exists an abelian variety $A_{\T^\max}$ such that
$\T^\max\cong\End(A_{\T^\max})$ and $\D_{\T^\max}$ is an equivalence of categories 
\[\AV_\Pi\to \TDel_{\Pi,L,H}.\]
 \end{thm}

\subsection{Supersingular Deligne modules.}
We now consider an application of the general equivalence theorem to supersingular abelian varieties.

\begin{deff}
Let $\Pi$ be a Galois invariant set of supersingular Weil numbers. 
Let $L=\QQ(\sqrt{-s})$ be an imaginary quadratic extension of $\QQ$, and let $g\in\Gal(L/\QQ)$ be the non-trivial element. 
{A supersingular Deligne module} $D$ with Weil support in $\Pi$ is a finitely generated  $R_\Pi\otimes_\ZZ\OO_L$-module with a semilinear operator $G$ such that $D$ is locally free over $\OO_L$, and for any $x\in \OO_L$ and $y\in D$ we have
\[G^2=-p\quad \text{and}\quad G(xy)=g(x)G(y).\]  
\end{deff}

Choose $\beta\in\OO_L$ such that $\beta^2=-s$, and $\alpha\in\ZZ$ such that 
$\alpha^2+p\equiv 0\pmod{s}$. Put \[H=\beta(\alpha-G)/s\in\HH_{\Pi,L}.\]

Denote the category of supersingular Deligne modules with Weil support in $\Pi$ that are invariant under $H$ by $\SSDel_{\Pi,L,H}$.

\begin{thm}\label{ss_main}
Let $\Pi$ be a Galois invariant set of supersingular Weil numbers over $\FF_q$. 
There exists a quadratic extension $L/\QQ$, where $s\equiv 3\pmod{4}$, such that 
$A\mapsto (\D_\T(A),M_*(A))$ is an equivalence from $\AV_\Pi$ to the product category 
\[\SSDel_{\Pi,L,H}\locp\times_{\DieW_\Pi\locp}\DieW_\Pi.\]
        
If $n=2$ and $F_\Pi=\pm p$, then $\D_\T$ is an equivalence of categories
from $\AV_\Pi$ to $\SSDel_{\Pi,L,H}$.
\end{thm}

\begin{rem}
    In his paper on maximal orders in quaternion algebras, Ibukiyama~\cite{Ib} constructed orders $O(s,\alpha)$ and $O'(s,\alpha')$, where $s$ is a prime such that 
    $s\equiv 3\pmod{8}$ and $\leg{-s}{p}=-1$. He proved that for any supersingular curve $E$ over $\FF_{p^2}$ with a model over $\FF_p$ either $\End(E)\cong O(s,\alpha)$ or $\End(E)\cong O'(s,\alpha')$.

    In our construction, if $p>2$, then $s\equiv 7\pmod{8}$.
    We can also start with a prime $s$ such that $s\equiv 3\pmod{8}$ and 
    $\leg{-s}{p}=-1$. Then the order $\T$ generated by $\T_{\ZZ,L}(-p,p)$, and 
    \[H=\beta(\alpha-G)/s\in\HH_{\Pi,L}\] is equal to the order $O(s,\alpha)$.
    Indeed, the following generators of the maximal order $O(s,\alpha)$ belong to $\T$:
    \[1,\frac{1+\beta}{2},\frac{(1+\beta)}{2}G, H.\]
\end{rem}

\begin{rem}
Let $E$ be an elliptic curve over $\FF_{p^2}$ such that $\End(E)$ is a $\ZZ$-module of rank $4$. In~\cite{EC18} Jordan, Keeton, Poonen, Rains, Shepherd-Barron and Tate proved that there is an anti-equivalence from the category of abelian varieties isogenous to a power of $E$ over $\FF_{p^2}$ to the category finitely presented torsion-free left $\End(E)$-modules.
Theorem~\ref{ss_main} is a generalization of this result.
\end{rem}

%\bigskip

The author thanks Tomoyoshi Ibukiyama and Uzi Vishne for their remarks on the paper.
The LLM Claude Opus 5.5 was used for proofreading. The author wrote the final text and assumes full responsibility for its content.

\section{Preliminaries}
In this section we recall some well-known facts about semi-simple algebras and abelian varieties over finite fields. In this paper the word \emph{module} always means a finitely generated module. For a ring $\HH$ we denote the opposite ring by $\HH^\op$, and $r\times r$-square matrices over $\HH$ by $\Mat_r(\HH)$. If $\HH$ is a $\ZZ$-algebra, we write $\HH_\ell$ for $\HH\otimes_\ZZ\ZZ_\ell$. If $K$ is a number field, we denote its ring of integers by $\OO_K$.

\subsection{Brauer groups}
We briefly recall the definition of the Brauer group $\Br(K)$ of a field $K$.
According to the Wedderburn Structure Theorem~\cite[Theorem 7.4]{MO}, every simple algebra is isomorphic to a matrix algebra $\Mat_r(\HH)$ over a division ring $\HH$. We say that two simple algebras $\HH_1$ and $\HH_2$ with center $K$ are ~\emph{similar} if they are isomorphic to matrix algebras over the same division ring $\HH$. Equivalently, there exist $r_1$ and $r_2$ such that \[\Mat_{r_2}\HH_1\cong\Mat_{r_1}(\HH_2).\] The set of similarity classes of simple algebras with center $K$ forms an abelian group $\Br(K)$ relative to the tensor product of algebras~\cite[Theorem 28.2]{MO}. The inverse element to the class of an algebra $\HH$ is represented by the opposite algebra $\HH^\op$.

For any simple algebra $\HH$ over $K$ there exists a field $L$ over $K$ such that 
\[L\otimes_K \HH\cong \Mat_r(L)\] for some $r$~\cite[Theorem 7.15]{MO}. We say that $\HH$ is \emph{split} by $L$, and $L$ is \emph{a splitting field of $\HH$}. 
The group of similarity classes split by $L$ is called \emph{the relative Brauer group $\Br(L/K)$}. The group $\Br(L/K)$ is a subgroup of $\Br(K)$.

\subsection{Cyclic algebras}\label{cyclic_ring}
Let $K$ be a field, and let $L$ be a cyclic Galois algebra over $K$ with Galois group $\Gal(L/K)$, that is, $L$ is a product of finite field extensions of $K$ with an action of $\Gal(L/K)$ such that the natural morphism \[\Gal(L/K)\times_{\Spec K} \Spec L\to \Spec L\times_{\Spec K}\Spec L\] 
given by $g\times x\mapsto x\times gx$ is an isomorphism. For example, if $L_0/\QQ$ is a Galois extension of fields, then $L_0\otimes_\QQ\QQ_\ell/\QQ_\ell$ is a Galois extension of algebras with Galois group $\Gal(L_0/\QQ)$.

Let $L$ be a cyclic Galois algebra of degree $m$ over $K$ with Galois group generated by an element $g$. We have a multiplicative norm map $N_{L/K}:L^*\to K^*$ given by 
\[N_{L/K}(x)=\prod_{i=0}^{m-1} g^i(x).\]

\begin{thm}\label{Th30_4} Let $L$ be a cyclic Galois algebra of degree $m$ over $K$ with Galois group generated by an element $g$. Then 
    \begin{enumerate}
    \item $(L/K,g,\lambda)$ is a simple algebra;
        \item $(L/K,g,1)\cong \Mat_m(K)$;    
        \item $(L/K,g^s,\lambda^s)\cong (L/K,g,\lambda)$
        for all $s\in\ZZ$ such that $(s,m)=1$;
        \item $(L/K,g,\lambda_1)\cong (L/K,g,\lambda_2)$ 
        if and only if there exists $a\in L^*$ such that $\lambda_1=N_{L/K}(a)\lambda_2$.    
    \end{enumerate}
\end{thm}
\begin{proof}
This is an easy generalization of \cite[Theorem 29.6]{MO} and \cite[Theorem 30.4]{MO}.
See also \cite[Section 30.A]{KMRT}.
\end{proof}

Let $L$ be a cyclic Galois algebra of degree $m$ over a local field $K$. 
Since $K$ is a local field, $\OO_L$ is a free $\OO_K$-module of rank $m$.
Clearly, there is a natural number $a$ such that $L\cong (L')^a$, where $L'/K$ is a cyclic field extension. 

The group algebra $\OO_L[\Gal(L/K)]$ is the ring generated over $\OO_L$ as a module by 
$\Gal(L/K)$ with the following relations:
\[[g]^m=1 ,\text{ and }\; [g]x=g(x)[g],\]
for all $x\in \OO_L$, and $g\in \Gal(L/K)$.
The following result is an easy consequence of the Auslander-Goldman-Rim Theorem~\cite[Theorem 40.14]{MO}. %We sketch a proof for the sake of completeness.

\begin{lemma}\label{MaxOrder}
Let $g$ be a generator of $\Gal(L/K)$.
If $L'/K$ is unramified, then the natural action of \[\T=\OO_L[\Gal(L/K)]\] 
on $\OO_L$ induces an isomorphism
\[\T\cong \Mat_m(\OO_K).\]
In particular, $\T$ is a maximal order.
\end{lemma}

\subsubsection{Cyclic algebras over local fields.}
 Assume that $K$ is a finite field extension of $\QQ_p$. 
 According to~\cite[Theorem 31.8]{MO}, we have an isomorphism  \[\inv:\Br(K)\to \QQ/\ZZ.\]
 The element $\inv(\HH)\in \QQ/\ZZ$ is called \emph{the Hasse invariant} of a central simple algebra $\HH$ over $K$ .
 If $L/K$ is a cyclic unramified extension of local fields, 
 then $\Gal(L/K)$ is generated by the Frobenius automorphism $\sigma=\sigma_{L/K}$. 

\begin{prop}\label{Hasse_inv}\cite[31.7]{MO}
Let $L/K$ be a cyclic unramified extension of non-archimedian local fields. 
Then the Hasse invariant of the cyclic algebra $(L/K,\sigma,\lambda)$ is equal to $v(\lambda)/[L:K],$ where $v$ is the normalized valuation on $L$.
\end{prop}

\subsubsection{Cyclic algebras over number fields.}
Let $K$ be a number field, and let $v$ be a place of $K$. We denote the Hasse invariant of the local algebra $\HH\otimes_{K} K_v$ by $\inv_v\HH\in\QQ/\ZZ$. 
A central simple algebra over a number field $K$ is uniquely determined up to isomorphism by its local invariants at the places of $K$~\cite[Theorem 32.11]{MO}. 

\subsection{Adjoint functors and Morita equivalence}
Let $S$ and $\T$ be arbitrary rings. For any left $S$-module $Y$, and a
left $S$-module $M$ with a structure of a right $\T$-module there is 
a structure of a left $\T$-module on $\Hom_S(M,Y)$
thanks to the right $\T$-module structure on $M$:
for any $\phi\in\Hom_S(M,Y)$, $m\in M$, and $\lambda\in\T$ we have 
\[(\phi\lambda)(m)=\phi(m\lambda).\] 
Let $X$ be a left $\T$-module. There is a natural adjunction isomorphism:
\[\Hom_\T(X,\Hom_S(M,Y))\cong\Hom_S(M\otimes_\T X,Y).\]
%If $X=\Hom_S(M,Y)$, then the identity morphism on $X$ corresponds to a natural morphism 

\begin{thm}[Morita equivalence.]\label{Morita}\cite[Corollary 16.9]{MO}
    Let $S$ be a ring, and let $M$ be a projective finitely generated $S$-module.
    Assume that $M$ is a generator in the category of finitely generated $S$-modules.
    Put $\T=\End_S(M)$. Then the adjoint functors
    \[\Hom_{S}(M,-):\Md{S}\rightleftarrows\Md{\T}:M\otimes_{\T}-\]   
    are both equivalences of categories.
\end{thm}

\subsection{Tate--Honda theory.}\label{TH_theory}
Let $A$ be an abelian variety of dimension $g$ over $\FF_q$, and let
$A_t$ be the kernel of multiplication by $t$ in
$A(\FFF_q)$. Let \[T_\ell(A) = \plim A_{\ell^r}\] be the $\ell$-th Tate module of $A$, 
and let $V_\ell(A)=T_\ell(A)\otimes_{\ZZ_\ell}\QQ_\ell$ be the
corresponding vector space over $\QQ_\ell$. Then $T_\ell(A)$ is a
free $\ZZ_\ell$-module of rank $2g$. The Frobenius endomorphism
$F_A$ of $A$ acts on the Tate module by a semisimple linear
transformation~\cite[IV.19. Theorem 4]{Mum}. 

Let $\pi$ be an algebraic integer.  We say that $\pi$ is \emph{a Weil number over $\FF_q$},
if for any homomorphism $\tau:\QQ(\pi)\to\CC$ we have $|\tau(\pi)|=\sqrt{q}$. 
For example, the eigenvalues of $F_A$ are Weil numbers over 
$\FF_q$~\cite[IV.21. Theorem 4]{Mum}.

The characteristic polynomial \[f_A(t) = \det(t-F_A|V_\ell(A))\]
is called {\it the Weil polynomial of $A$}. It is a monic
polynomial of degree $2g$ with rational integer coefficients
independent of the choice of $\ell$. 
Tate proved that the isogeny class of an abelian variety is determined
by its characteristic polynomial, that is, $f_A(t)=f_B(t)$ implies
that $A$ is isogenous to $B$~\cite{Ta66}.

Let $\pi$ be a Weil number.
Define the invariant of $\pi$ at a place $v$ of $\QQ(\pi)$ as follows:
  \[\inv_v(\pi)=\left\lbrace \begin{array}{ll}
    1/2, & \hbox{if } v \text{ is real }; \\
    \frac{\ord_v(\pi)}{\ord_v(q)}[\QQ(\pi)_v:\QQ_p], & \hbox{if } v|p;\\
    0, & \text{otherwise.}\\
\end{array}\right.\]

Denote by $P_\pi\in\ZZ[t]$ the minimal monic polynomial of $\pi$. 
Then \emph{the Weil polynomial associated to $\pi$}  is $f_\pi=P_\pi^{e_\pi}$,
 where $e_\pi$ is the least common multiple of denominators of $\inv_v(\pi)$ for all $v$.
 According to the Honda Theorem~\cite{Ho}, $f_\pi$ is the Weil polynomial of a simple abelian variety $B_\pi$ over $\FF_q$ such that $\inv_v(\pi)$ are invariants of the simple algebra $\End^\circ(B_\pi)$ over $\QQ(\pi)$.

Let $\Pi$ be a Galois invariant set of Weil numbers.
The category $\AV_\Pi$ of abelian varieties with Weil support in $\Pi$ is the full subcategory of abelian varieties isogenous to a product of $B_\pi$, where $\pi\in\Pi$.
Choose a subset of representatives
$\Pi_0\subset \Pi$, that is, the natural morphism $\Pi_0\to \Pi/\Gal(\overline{\QQ}/\QQ)$ is a bijection. Let $P\in\ZZ[t]$ be a monic separable polynomial such that $P(0)\neq 0$. Then $t$ is invertible in the quotient algebra $\QQ[t]/P(t)\QQ[t]$. Define $R_P$ as the image of $\ZZ[t,q/t]$ in $\QQ[t]/P(t)\QQ[t]$. 
 Clearly, the ring $R_\pi=\ZZ[\pi,q/\pi]\subset\Bar\QQ$ is naturally isomorphic to $R_{P_\pi}$. 
 Denote by \[F_\Pi\in\prod_{\pi\in\Pi_0} R_{P_\pi}\] the element that projects to $t$ in $R_{P_\pi}$ for any $\pi\in\Pi_0$. Note that $F_\Pi$ does not depend on the choice of $\Pi_0$.
Put \[R_\Pi=\ZZ[F_\Pi,q/F_\Pi]\subset\prod_{\pi\in\Pi_0} R_{\pi}.\]

Let $A$ be an abelian variety with the Weil polynomial $f_A$, %=\prod_{\pi\in\Pi}f_\pi$.
and let $\Pi$ be the set of roots of $f_A$.
There is a natural injective morphism $R_\Pi\to\End A$ such that $F_\Pi\mapsto F_A$. Moreover, this morphism induces an isomorphism of $E_\Pi=R_\Pi\otimes\QQ$ with the center of $\End^\circ(A)$~\cite{Ta66}. 

\begin{thm}\cite{Ta66}\label{Tate_thm}
Let $A$ and $B$ be abelian varieties over $\FF_q$. Then
    \[\Hom(A,B)\otimes_\ZZ\ZZ_\ell\cong\Hom_{R_\Pi}(T_\ell(A),T_\ell(B)). \]
\end{thm}

Let $S$ be a finite commutative $R_\Pi$-algebra. We say that an abelian variety $A$ from $\AV_\Pi$ has an action of $S$ if there is a homomorphism $S\to \End(A)$ extending $R_\Pi\to \End(A)$. 
Define the category $\AV_{S}$ of $S$-abelian varieties as the category of abelian varieties from $\AV_\Pi$ with an $S$-action. The Tate module of an $S$-abelian variety is an $S$-module. Moreover, if $S\subset E_\Pi$, then $\End_{\AV_{S}}(A)\cong\End(A)$.

Let $N$ be a natural number.
Define the category of abelian varieties up to $N$-isogeny $\AV_{S}[\frac{1}{N}]$ as the category of abelian varieties from $\AV_S$, where the Hom-group is defined as 
\[\Hom_{\AV_{S}[\frac{1}{N}]}(A,B)=\Hom_{\AV_{S}}(A,B)[\frac{1}{N}].\]

\begin{lemma}\label{iso_lemma}
Let $S$ be a finite commutative $R_\Pi$-algebra.
    Let $B$ be an $S$-abelian variety, and let $T_\ell\subset T_{\ell}(B)$ be a family of $S$-submodules such that $T_\ell= T_{\ell}(B)$ for almost all $\ell$. Then there exist an $S$-abelian variety $A$ and an $S$-isogeny $\phi:A\to B$ such that $T_\ell(\phi)$ induces an isomorphism $T_{\ell}(A)\to T_\ell$ for all $\ell\neq p$.
\end{lemma}
\begin{proof}
Let $\ell_1, \dots,\ell_r$ be the finite set of primes such that 
$T_{\ell_i}\neq T_{\ell_i}(B)$. Put $B_{r+1}=B$. 
It follows from~\cite[IV.2.3]{Milne} that for any $1\leq i\leq r$ there exists an $\ell_i$-isogeny $B_i\to B_{i+1}$ such that $T_{\ell_i}(B_i)\cong T_{\ell_i}$.
Let $A=B_1$, and let $\phi$ be the composition of $\phi_i$. Then $T_{\ell}(A)\cong T_\ell$ for all $\ell\neq p$. The lemma is proved.
\end{proof}

\subsection{\Die modules of $S$-abelian varieties.}
Let $W=W(\FF_q)$ be the ring of Witt vectors over $\FF_q$. We denote by $\sigma: W\to W$ the lift of the Frobenius automorphism of $\FF_q$.
The localization $W\locp$ is the field of fractions of $W$.

The \Die ring $D_{\FF_q}$ is the ring generated over $W$ by formal variables $F$ and $V$ such that $FV=VF=p$, and for any $a\in W$ we have:
\[ Fa=\sigma(a)F,\text{ and } aV=V\sigma(a). \]
There is an anti-equivalence of categories $M^*(-)$ between finite commutative group schemes over $\FF_q$ of $p$-power order and left $D_{\FF_q}$-modules of finite 
length~\cite[Theorem 28.3]{Pink}. 

The $p^r$-torsion $A[p^r]$ of an abelian variety $A$ over $k$ is a finite group scheme, and the limit \[M^*(A)=\varprojlim_r M^*(A[p^r])\] is called \emph{the contravariant \Die module of $A$}. The module $M^*(A)$ is free over $W$ of rank $2\dim A$. 

The \Die modules of $S$-abelian varieties can be defined as objects of a localized category of \Die modules, but we use a more straightforward approach. 
Let  \[D^\circ_\Pi=(E_\Pi\otimes_{\QQ} W\locp,\sigma,F_\Pi)\]
be the cyclic algebra generated over $E_\Pi\otimes_\QQ W\locp$ by a variable $F$ 
such that for any $a\in W\locp$ we have: 
\[F^n=F_\Pi, \text{ and } Fa=\sigma(a)F.\] 
The algebra $D_S$ is the order in $D^\circ_\Pi$ generated over $S\otimes_\ZZ W$ by $F$ and $V=p/F$.
There is a surjective homomorphism of rings $S\otimes_{\ZZ}D_{\FF_q}\to D_{S}$, and for any $S$-abelian variety $A$ the action of $S\otimes_{\ZZ}D_{\FF_q}$ on $M^*(A)$ factors through $D_{S}$.

The module $M(A)=\Hom_\ZZ(M^*(A),\ZZ_p)$ is naturally a $D_S^\op$-module. 
We get a covariant functor \[M:\AV_{S}\to\TF{D_{S}^\op},\]
where $\TF{D_S^\op}$ is the category of finitely generated 
 $D_{S}^\op$-modules that are torsion-free as $\ZZ_p$-modules. 
Clearly, the functor $M\mapsto\Hom_{\ZZ_p}(M,\ZZ_p)$
is an anti-equivalence from the category of \Die modules to $\TF{D_S^\op}$.

\begin{rem}\label{die_dual}
\emph{The covariant \Die module of $A$} is defined as \[M_*(A)=\Hom_{W}(M^*(A),W).\]
%The category covariant $S$-\Die modules
Let $W^*=\Hom_{\ZZ_p}(W,\ZZ_p)$. Fix an isomorphism $W\to W^*$. For any 
contravariant \Die module $M$ the adjunction isomorphism induces an isomorphism of $D_S^\op$-modules:
    \[\Hom_{W}(M,W)\cong\Hom_{W}(M,\Hom_{\ZZ_p}(W,\ZZ_p))\cong \Hom_{\ZZ_p}(M,\ZZ_p).\]
\end{rem}

From the main result of~\cite{WM} we obtain a counterpart to Theorem~\ref{Tate_thm} for $D_S^\op$-modules.

\begin{thm}\cite{WM}\label{WM_thm}
Let $A$ and $B$ be abelian varieties over $\FF_q$. Then
    \[\Hom(A,B)\otimes_\ZZ\ZZ_p\cong\Hom_{D_S^\op}(M(A),M(B)).\]
\end{thm}

If $f:A\to B$ is an isogeny of $S$-abelian varieties, 
then $M^*(f)$ is injective and $M=M^*(f)(M^*(B))$ is a $D_S$-submodule of $M^*(A)$. 
We have the following analog of Lemma~\ref{iso_lemma}. 

\begin{lemma}
\label{lem_on_Dmod}
If $M\subset M^*(A)$ is a $D_S$-submodule such that 
$M^*(A)\locp\cong M\locp$, then there exists an $S$-abelian variety $B$ over $\FF_q$ and a $p$-isogeny $f:A\to B$ such that $M^*(f)$ induces an isomorphism $M^*(B)\cong M$.
\end{lemma}
\begin{proof}
The quotient $M^*(A)/M$ corresponds to an $S$-stable group subscheme $N\subset A$. Put $B=A/N$.
 \end{proof}

\section{A general equivalence theorem}\label{general_sec}
\subsection{Orders in $\HH_{\Pi,L}(\lambda)$.}
Let $S\subset E_\Pi$ be a finite $R_\Pi$-algebra.

\begin{notation}\label{H_notation}
Fix a cyclic extension $L/\QQ$ of degree $m$ such that $p$ is unramified.
Put \[n_0=m/[\Hat{L}_p:\QQ_p].\] Choose a generator $g$ of $\Gal(L/\QQ)$ such that $g^{n_0}$ is the Frobenius automorphism at $p$. Choose $\lambda\in S$.
Let \[\HH_{\Pi,L}(\lambda)=\oplus_{\pi\in\Pi_0}\HH_{\pi,L}(\lambda),\] 
where $\HH_{\pi,L}(\lambda)$ is the cyclic algebra $(L(\pi)/\QQ(\pi),g,\lambda)$.  
\end{notation}
 
Note that $\HH_{\Pi,L}(\lambda)$ is a $E_\Pi$-algebra with an injective homomorphism $L\to\HH_{\Pi,L}(\lambda)$.

\begin{deff}
Let $\T\subset\HH_{\Pi,L}(\lambda)$ be an order such that $S\otimes_\ZZ\OO_L\subset\T$.
\begin{itemize}
    \item \emph{A tame structure} over $\T$ is a collection of $\T^\op_\ell$-modules 
    $\{T_{\ell,\T}\}_{\ell\neq p}$ such that for any $\ell\neq p$ 
    there is an isomorphism of $S_\ell\otimes_\ZZ\OO_L$-modules 
\[T_{\ell,\T}\cong S_\ell\otimes_\ZZ\OO_L.\]
%    the module $T_{\ell,\T}$ is projective over $S_\ell$, and 
 %   \[T_{\ell,\T}\otimes_{\ZZ_\ell}\QQ_\ell\cong E_\Pi\otimes_\QQ L.\]
    %, if there is a structure of a $\T^\op_\ell$-module on $S_\ell\otimes_\ZZ\OO_L$.
\item A $\T^\op_\ell$-module $T_{\ell,\T}$ is \emph{$\ell$-balanced} if the action on $T_{\ell,\T}$ induces an isomorphism
\[\T^\op_\ell\cong\End_{S_\ell}(T_{\ell,\T}).\]
\item A tame structure $\{T_{\ell,\T}\}_{\ell\neq p}$ is \emph{balanced} if any $T_{\ell,\T}$ is $\ell$-balanced.
\end{itemize}
\end{deff}

\begin{deff}
Let $\T\subset\HH_{\Pi,L}(\lambda)$ be an order such that $S\otimes_\ZZ\OO_L\subset\T$.
We say that $\T$ is \emph{$\ell$-cyclic} if $\T_\ell=\T_\ell(\lambda)$, where $\T_\ell(\lambda)$ is generated over $S_\ell\otimes_\ZZ\OO_L$ by $G_\lambda$.
\end{deff}

\begin{ex}\label{L_order}
For any $\rho\in S\otimes_\ZZ\OO_L$ there is an order
$\T_{S,L}(\lambda,\rho)\subset\HH_{\Pi,L}(\lambda)$ generated over $S\otimes_\ZZ \OO_L$ by 
$G_\lambda$ and $\CG_\lambda=\rho G_\lambda^{-1}$.  
\end{ex}

\begin{rem}\label{l_cyclic}
If $\lambda\in S^*_\ell$ is invertible, then $\T_{S,L}(\lambda,\rho)$ is $\ell$-cyclic.    
\end{rem}

We have a multiplicative norm map $N_L:(S\otimes_\ZZ L)^*\to (S\otimes_\ZZ\QQ)^*$ given by 
\[N_L(x)=\prod_{i=0}^{m-1} g^i(x).\]

\begin{lemma}\label{tame_module}
If $\T$ is $\ell$-cyclic, then there is a structure of a $\T^\op$-module on $S_\ell\otimes_\ZZ\OO_L$ if and only if 
there exists $\omega_\ell\in S_\ell\otimes_\ZZ\OO_L$ such that 
$\lambda=N_L(\omega_\ell)$.
\end{lemma}
\begin{proof}
    If $\lambda$ is a norm of $\omega_\ell\in S_\ell\otimes_\ZZ\OO_L$, then 
    $S_\ell\otimes_\ZZ\OO_L$ has a structure of a $\T^\op_\ell$-module as follows:
\[G_\lambda(x\otimes y)=\omega_\ell(x\otimes g^{-1}(y)).\]

Assume that there is a structure of a $\T^\op_\ell$-module on $S_\ell\otimes_\ZZ\OO_L$. Put
\[\omega_\ell= G_\lambda(1\otimes 1)\in S_\ell\otimes_\ZZ\OO_L.\]
Clearly, the norm of $\omega_\ell$ is equal to
$N_L(\omega_\ell)=G_\lambda^m(1\otimes1)=\lambda.$
\end{proof}

\begin{prop}\label{balanced_order}
Let $\T$ be an $\ell$-cyclic order in $\HH_{\Pi,L}(\lambda)$.
Assume that $\lambda\in S_\ell^*$ is invertible.
Then there is an $\ell$-balanced structure of a $\T^\op$-module on $S_\ell\otimes_\ZZ\OO_L$
%Then $\T$ is balanced 
in the following cases:
\begin{enumerate}
    \item $S$ is maximal at any ideal over $\ell$, and $\ell$ is unramified in $L$;
    \item $L$  splits completely at $\ell$.
\end{enumerate}
\end{prop}
\begin{proof}
First, we show that there exists $\omega_\ell\in S_\ell\otimes_\ZZ\OO_L$ such that $\lambda=N_L(\omega_\ell)$.

If $S_\ell$ is maximal, then $S_\ell\cong\prod_{\p|\ell} S_\p$ is a product of discrete valuation rings. Let $\lambda=\sum_\p \lambda_\p$ be the corresponding decomposition of $\lambda$, and $E_\Pi\otimes_\QQ\QQ_\ell\cong\prod_{\p|\ell} E_\p$, where
$E_\p\cong S_\p\otimes_{\ZZ_\ell}\QQ_\ell$.
    The $\QQ_\ell$-algebra 
    \[E_\p\otimes_{\QQ_\ell} L_\ell=\prod_{\q|\p} L_{\q}\] is a product of finite extensions of $E_\p$. 
If $L_\ell$ is unramified over $\QQ_\ell$, then $\lambda_\p$ is a norm of some unit in $L_\q$~\cite[Corollary on p. 29]{CF}, that is, 
    \[\lambda_\p=N_{L_{\q}/E_\p}(\omega_{\q})\] 
    for some $\omega_{\q}\in S_\p\OO_{L_\q}.$ 
    For any ideal $\p$ of $S_\ell$ we choose one $\q_0$ over $\p$ and put 
    \[\omega_\ell=\oplus_\p (\omega_{\q_0},1,\dots,1)\in\oplus_\p\oplus_\q S_\p\OO_{L_\q}. \]
    It is straightforward to check that $\lambda$ is a norm of $\omega_\ell$. 

 If $L$ splits completely at $\ell$, then $\OO_L\otimes_\ZZ\ZZ_\ell\cong\ZZ_\ell^m$, and
    $\lambda$ is a norm of an invertible element
    \[\omega_\ell=(\lambda,1,\dots,1)\in S_\ell^m\cong S_\ell\otimes_\ZZ\OO_L.\]
    
In both cases, there is an isomorphism 
\[\T_\ell\cong S_\ell\otimes_{\ZZ_\ell}\OO_L[\Gal(L/\QQ)].\]    
given by $G_\lambda\mapsto \omega_\ell (1\otimes g)$.
 Now apply Lemma~\ref{MaxOrder}.
\end{proof}

\begin{deff}
Assume that there exists a tame structure $\{T_{\ell,\T}\}_{\ell\neq p}$ over $\T$.
\emph{The Tate module} of a
$\T^\op$-module $D$ is the $S_\ell$-module
\[\cT_\ell(D)=\Hom_{\T^\op}(T_{\ell,\T},D\otimes_\ZZ\ZZ_\ell).\]
We get a functor:
\[\cT_\ell:\Md{\T^\op}\to\Md{S_\ell}.\]   
\end{deff}

\begin{rem}
%Assume that $T_{\ell,\T}\cong S_\ell\otimes_\ZZ\OO_L$.
    If we choose $\omega_\ell\in S_\ell\otimes_\ZZ\OO_L$ as in Lemma~\ref{tame_module}, then
    \[\cT_\ell(D)\cong (D\otimes_\ZZ\ZZ_\ell)^{G_\lambda=\omega_\ell}.\]
\end{rem}

There is a pair of adjoint functors 
\[\hT_\ell:\Md{\T^\op_\ell}\rightleftarrows\Md{S_\ell}:D_\ell\]   
given by 
\[\hT_\ell(D)=\Hom_{\T^\op_\ell}(T_{\ell,\T},D), \text{ and}\]
\[D_\ell(T)=T\otimes_{S_\ell}T_{\ell,\T}.\]

\begin{deff}\label{l_rigid}
Assume that there exists a tame structure $\{T_{\ell,\T}\}_{\ell\neq p}$ over $\T$.
Let $\ell\neq p$ be a prime number.
We say that a $\T^\op$-module $D$ is \emph{$\ell$-rigid}, if
the adjunction morphism 
\[D_\ell(\cT_\ell(D))\to D\otimes_\ZZ\ZZ_\ell\]
is an isomorphism.
\end{deff}

Denote by $\TF{S_\ell}$ the category of $S_\ell$-modules that are torsion-free as $\ZZ_\ell$-modules, and by $\TF{\T^\op}$ the category of left $\T^\op$-modules that are locally free over $\OO_L$.

\begin{prop}\label{Tate_eq}
Assume that there exists a tame structure $\{T_{\ell,\T}\}_{\ell\neq p}$ over $\T$.
\begin{enumerate}
   \item For any $T$ from $\TF{S_\ell}$ the adjunction morphism 
   \[T\to \hT_\ell(D_\ell(T))\] is an isomorphism.
    \item If $T_{\ell,\T}$ is $\ell$-balanced, then $\hT_\ell$ is an equivalence of categories, and any object of $\TF{\T^\op}$ is $\ell$-rigid. 
 \end{enumerate}
    \end{prop}
\proof
\begin{enumerate}
\item The algebra $\T_\ell$ is a subalgebra of 
$\End_{S_\ell}(S_\ell\otimes\OO_L)$ of finite index.
Since $T$ is torsion-free as a $\ZZ_\ell$-module, by Theorem~\ref{Morita} we get
\[\hT_\ell(D_\ell(T))=
\Hom_{\T^\op}(T_{\ell,\T},T\otimes_{S_\ell}T_{\ell,\T})\cong\]
\[\cong\Hom_{\End_{S_\ell}(S_\ell\otimes\OO_L)^\op}(T_{\ell,\T},
T\otimes_{S_\ell}T_{\ell,\T})\cong T.\]
\item Since $T_{\ell,\T}$ is a projective $S_\ell$-module, this follows from Theorem~\ref{Morita}. \qed
\end{enumerate}

\subsection{Quasi-free \Die modules.}
Assume that we have a structure of a left $\T^\op$-module on a contravariant \Die module $M_{\T}$, that is, $M_{\T}$ has a structure of a right $D_S^\op$-module.

\begin{deff}
The \emph{covariant \Die functor}    
\[\M:\TF{\T^\op}\to \Md{D_{S}^\op}\] is defined by the formula:
\[D\mapsto \Hom_{\T_p^\op}(M_{\T},D\otimes_\ZZ\ZZ_p).\] 
\end{deff}

There is a pair of adjoint functors 
\[\hM:\TF{\T^\op}\rightleftarrows\Md{D_{S}^\op}:D_p\]   
given by 
\[\hM(D)=\Hom_{\T_p^\op}(M_{\T},D), \text{ and}\]
\[D_p(M)=M_{\T}\otimes_{D_S^\op}M.\]

\begin{deff}\label{q_free}
We say that $M_{\T}$ is \emph{quasi-free over $\T^\op$} if the natural morphism 
\[M\to \hM(D_p(M))\] is an isomorphism for any $D_S^\op$-module $M$.

We say that a left $\T^\op$-module $D$ is \emph{$p$-rigid}, 
if $D\otimes_\ZZ\ZZ_p$ is torsion free over $\ZZ_p$ and the natural adjunction morphism 
\[ D_p(\M(D))\to D\otimes_\ZZ\ZZ_p\] is an isomorphism. 
A module $D$ is \emph{rigid} if it is $\ell$-rigid for all primes $\ell$.
Denote the subcategory of rigid objects in $\TF{\T^\op}$ by $\RM{\T^\op}$.
\end{deff}

\begin{lemma}\label{quasi_free}
Assume that $\HH_{\Pi,L}(\lambda)\otimes_\QQ\QQ_p\cong\End_{D_S}^\circ(M_{\T})^\op$.
If $M_{\T}$ is free as a $D_S$-module, then $M_{\T}$ is quasi free.
\end{lemma}
\begin{proof}
      The algebra $\T_p^\op$ is a subalgebra of the matrix algebra 
   $\HH=\End_{D_S}(M_{\T})$ of finite index,
   and for any $D_S^\op$-module $M$ the structure of a $\T_p^\op$-module on 
   \[D_p(M)=M_{\T}\otimes_{D_S^\op}M\] is induced from the structure of an $\HH$-module. 
   Hence, we have an isomorphism 
\[\hM(D_p(M))=\Hom_{\T_p^\op}(M_{\T},M_{\T}\otimes_{D_S^\op}M)\to 
\Hom_{\HH}(M_{\T},M_{\T}\otimes_{D_S^\op}M).\]
Since $\HH$ is Morita equivalent to $D_S$, we have $M\cong \hM(D_p(M))$, and $M_{\T}$ is quasi-free.
\end{proof}

We say that $\T$ is \emph{$p$-balanced} if $\T_p^\op\cong\End_{D_S}(M_{\T})$ for some quasi-free \Die module $M_{\T}$. We get an immediate corollary of Theorem~\ref{Morita}.

\begin{prop}\label{p_rigid_prop}
Assume that $\T$ is $p$-balanced.
Then $\hM$ is an equivalence of categories, and any object of $\TF{\T^\op}$ is $p$-rigid.\qed
 \end{prop}

\subsection{Wide abelian varieties}
Assume that there exists an $S$-abelian variety $A_{\Pi,L}$ such that 
\[h_{\Pi,L}:\HH_{\Pi,L}(\lambda)\cong \End^\circ(A_{\Pi,L}).\]
Let $\T\subset \HH_{\Pi,L}(\lambda)$ be an order such that $S\otimes_\ZZ\OO_L\subset \T$.
If $\{T_{\ell,\T}\}_{\ell\neq p}$ is a tame structure over $\T$, then for any $\ell\neq p$
the module $T_{\ell,\T}^*=\Hom_{\ZZ_\ell}(T_{\ell,\T},\ZZ_\ell)$ is a $\T_\ell$-module.

\begin{deff}\label{wide_AV_def}
A pair of an $S$-abelian variety $A_\T$ and an injective homomorphism of $S$-algebras
\[h_{A_\T}:\T\to\End_S(A_\T)\] is \emph{wide} 
if $\{T^*_\ell(A_\T)\}_{\ell\neq p}$ is a tame structure over $\T$. 
\end{deff}

\begin{rem}\label{h_iso}
   If $A_{\T}$ is a wide abelian variety over $\T$, then any Tate module $T_\ell(A)$
   is a $\T$-module, and $h_A$ induces an isomorphism
\[h_A^\circ:\HH_{\Pi,L}(\lambda)\to\End^\circ(A_{\T}).\]
Indeed, $h_A^\circ$ is injective and the dimensions of both algebras are equal to 
$m^2\dim_\QQ E_\Pi.$ 
\end{rem}

\begin{prop}\label{wide_AV}
Let $\T\subset \HH_{\Pi,L}(\lambda)$ be an order with a tame structure 
$\{T_{\ell,\T}\}_{\ell\neq p}$.
\begin{enumerate}
    \item Then there exists a wide abelian variety $A_{\T}$ over $\T$ such that 
    $T_{\ell,\T}\cong T^*_\ell(A_\T)$ for all $\ell\neq p$.
    \item If the tame structure $\{T_{\ell,\T}\}_{\ell\neq p}$ is balanced, then 
    \[\T\locp\cong\End(A_\T)\locp. \]
\item If there exists a quasi-free contravariant \Die module $M_{\T}$, then there 
exists a wide abelian variety $A_{\T}$ over $\T$ such that $M^*(A_\T)\cong M_\T$ is quasi-free.
\item If the tame structure $\{T_{\ell,\T}\}_{\ell\neq p}$ is balanced,
$M^*(A_\T)$ is quasi-free, and $\T$ is $p$-balanced, then
    \[\T\cong\End(A_\T). \]
\end{enumerate}
\end{prop}
\begin{proof}
There exists $N\in\NN$ such that $N h_{\Pi,L}(\T)\subset\End(A_{\Pi,L})$.
Assume that $\ell$ divides $N$, and the order $\T$ is maximal at $\ell$.
Then the Tate module $T_{\ell}(A_{\Pi,L})$ is a $\T_\ell$-module and all $\T_\ell$-lattices in
\[\Hom_{\QQ_\ell}(V_{\ell}(A_{\Pi,L}),\QQ_\ell)\cong V_{\ell,\T}=T_{\ell,\T}\otimes_{\ZZ_\ell}\QQ_\ell,\]
 are isomorphic. Hence, there is an isomorphism of $\T$-modules
\[\Hom_{\ZZ_\ell}(T_{\ell}(A_{\Pi,L}),\ZZ_\ell)\cong T_{\ell,\T}.\]

For the finite number of remaining primes $\ell\neq p$ we choose
an injection $T_{\ell,\T}^*\to V_\ell(A_{\Pi,L})$ of $\T$-modules,
and we may and will assume that, under this injection, 
$T_{\ell,\T}^*$ goes to a submodule of $T_\ell(A_{\Pi,L})$.
According to Lemma~\ref{iso_lemma}, there exists an isogeny $A'\to A_{\Pi,L}$ such that 
$T_\ell(A')\cong T_{\ell,\T}^*$. 
According to Lemma~\ref{lem_on_Dmod}, there exists a $p$-isogeny $A_{\T}\to A'$ such that 
$\T M^*(A')\cong M^*(A_\T)$ is $\T$-invariant.

We now prove that $h_{\Pi,L}$ factors through $h_A:\T\to\End(A_{\T})$.
Consider the natural inclusion $\iota:\End(A_{\T})\to U$ to the $\ZZ$-submodule generated by 
$\End(A_{\T})$ and $h_{\Pi,L}(\T)$. 
For all primes $\ell$ the localization $\iota\otimes\ZZ_\ell$ is an isomorphism because, according to Tate's Theorem, 
\[\End(A_{\T})\otimes_\ZZ\ZZ_\ell\cong\End_{S_\ell}(T_{\ell}(A_\T)), \text{ and }\] 
\[\End(A_{\T})\otimes_\ZZ\ZZ_p\cong\End_{D_S}(M^*(A_\T))^\op\] 
are invariant under the action of $\T$.
This implies that $U=\End(A_{\T})$, and $h_{\Pi,L}$ factors through $h_A$. We proved that $A_{\T}$ is wide. 

Assume that $\{T_{\ell,\T}\}_{\ell\neq p}$ is a balanced tame structure.
It follows from the Tate's theorem that for any $\ell\neq p$ there is an isomorphism 
\[\End(A_{\T})\otimes_\ZZ\ZZ_\ell\cong\End_{S_\ell}(T_{\ell,\T});\] 
therefore, $h_{A_\T}\otimes\ZZ_\ell$ is an isomorphism for all $\ell\neq p$. This proves $(2)$.

Suppose that there exists a quasi-free contravariant \Die module $M_{\T}$.
Choose an injection $M_{\T}\to M^*(A_\T)$.
According to Lemma~\ref{lem_on_Dmod}, there exists a $p$-isogeny $A_\T\to A_{\T}'$ such that 
$M_{\T}\cong M^*(A_\T')$. Part $(3)$ is proved. 

Finally, $(4)$ follows from $(2)$, and the $p$-adic Tate's theorem.
\end{proof}

\subsection{Products of categories}\label{product_sec}
First, recall the definition of the product of categories.
Suppose that we have two functors $\alpha_1:\C_1\to\C$ and $\alpha_2:\C_2\to \C$ to a category $\C$. By definition, the product  of $\C_1$ and $\C_2$ over $\C$ is the category  of triples $(c_1,c_2,a)$, where $c_1$ and $c_2$ are objects of $\C_1$ and $\C_2$ respectively, and $a:\alpha_1(c_1)\to \alpha_2(c_2)$ is an isomorphism in $\C$. A morphism from $(c_1,c_2,a)$ to $(c'_1,c'_2,a')$ is a pair of morphisms $a_i:c_i\to c'_i$ in $\C_i$ for $i\in\{1,2\}$ such that the following diagram commutes:
\[\xymatrix{
\alpha_1(c_1)\ar[d]_{a} \ar[r]^{\alpha_1(a_1)} & \alpha_1(c'_1)\ar[d]_{a'} \\
\alpha_2(c_2) \ar[r]^{\alpha_2(a_2)} & \alpha_2(c'_2). }\]

We get a diagram of categories and functors:
\[\xymatrix{
\C_1\times_\C \C_2 \ar[r] \ar[d] & \C_1  \ar[d] \\
\C_2 \ar[r] & \C } \]

We now define the category $\cC_{S,L}(\T)$. Let $A_\T$ be a wide abelian variety over $\T$.
The contravariant \Die module $M_{\Pi,L}=M^*(A_\T)\locp$ is a left
$\HH_{\Pi,L}(\lambda)^\op$-module.
The formula \[D\mapsto \Hom_{\HH_{\Pi,L}^\op}(M_{\Pi,L},D\otimes_\ZZ\QQ_p)\]
defines the functor
\[\M_\QQ:\TF{\T^\op}\locp\to\Md{D_{S}^\op\locp}.\] 
According to Theorem~\ref{Morita}, 
there is an equivalence of categories
\[\M^\circ:\Md{\T_p^\op\locp}\to\Md{D_{\Pi}^\op},\] where  
$\M^\circ(D)=\Hom_{\T_p^\op\locp}(M_{\Pi,L},D).$

On the other hand, the category of $D_{S}^\op$-modules has a natural forgetful functor to the corresponding category over $\QQ_p$: \[\TF{D_S^\op}\to \Md{D_{S}^\op\locp}.\]  
We define the product category using this data:
\[\cC_{S,L}(\T)=\TF{\T^\op}\locp\times_{\Md{D_{S}^\op\locp}}\TF{D_S^\op}.\]

 Explicitly, $\cC_{S,L}(\T)$ can be described as follows.
Objects of $\cC_{S,L}(\T)$ are triples $(D,M,\alpha)$, where
\begin{itemize}
	\item $D$ is a left $\T^\op\locp$-module that is locally free over $\OO_L\locp$;
	\item $M$ is a left $D_{S}^\op$-module;
	\item $\alpha$ is an isomorphism of left $D_S^\op$-modules:
	\[\alpha:\M_\QQ(D)\cong M\locp.\]
\end{itemize}

A morphism of triples $(D_1,M_1,\alpha_1)$ and $(D_2,M_2,\alpha_2)$ is a pair of morphisms 
$(\phi_D,\phi_M):(D_1,M_1)\to (D_2,M_2)$ such that 
\[\phi_M\locp\circ\alpha_1=\alpha_2\circ\M_\QQ(\phi_D).\]

\begin{prop}\label{full_faithful}
Let $(D_1,M_1,\alpha_1)$ and $(D_2,M_2,\alpha_2)$ be objects of the product category 
$\cC_{S,L}(\T)$, and let $\ell\neq p$. Then 
\[\Hom_{\C_{S,L}(\T)}((D_1,M_1,\alpha_1),(D_2,M_2,\alpha_2))\locp
\cong\Hom_{\T^\op}(D_1,D_2)\locp,\] and 
\[\Hom_{\C_{S,L}(\T)}((D_1,M_1,\alpha_1),(D_2,M_2,\alpha_2))
\otimes\ZZ_p \cong\Hom_{D_S^\op}(M_1,M_2).\]
\end{prop} 
\begin{proof}
Clearly, the first morphism is injective. 
It is enough to prove the surjectivity on generators of the form
\[a\phi\in \Hom_{\T^\op}(D_1,D_2)\locp,\] where
$\phi: D_1\to D_2$  is a $\T^\op$-morphism, and $a\in\ZZ\locp$. 
The composition \[\phi_p=\alpha_2\circ \M_\QQ(\phi_D)\circ\alpha_1^{-1}\]
is an element of $\Hom_{D_S^\op}(M_1\locp, M_2\locp)$. 
 There exists $r\in\NN$ such that $p^r\phi_p(M_1)\subset M_2$, therefore
$a\phi$ is an image of \[\frac{a}{p^r}(p^r\phi,p^r\phi_p)\in
\Hom_{\C_{S,L}(\T)}((D_1,M_1,\alpha_1),(D_2,M_2,\alpha_2))\locp.\]

Let us prove the second assertion. Any element of
\[\Hom_{\C_{S,L}(\T)}((D_1,M_1,\alpha_1),(D_2,M_2,\alpha_2))\otimes\ZZ_p \]
is a pair $(\phi_D,\phi_M)$, where $\phi_M\in\Hom_{D_S^\op}(M_1,M_2)$, and 
$\phi_D\in\Hom_{\T^\op\locp}(D_1,D_2)\otimes_{\ZZ\locp}\QQ_p$ such that
\[\phi_M\locp\circ\alpha_1=\alpha_2\circ\M_\QQ(\phi_D).\]

For any $\phi_M\in \Hom_{D_S^\op}(M_1,M_2)$ the morphism
\[\Tilde\phi_D=\alpha_2^{-1}\circ\phi_M\locp\circ\alpha_1\]
is an element of
\[\Hom_{D_{S}^\op}(\M_\QQ(D_1),\M_\QQ(D_2))\cong 
\Hom_{D_{S}^\op}(\M^\circ(D_1\otimes_{\ZZ}\QQ_p),\M^\circ(D_2\otimes_{\ZZ}\QQ_p))\cong\] 
\[\cong\Hom_{\T^\op\locp}(D_1,D_2)\otimes_{\ZZ\locp}\QQ_p,\]
where the last isomorphism comes from the Morita equivalence. Hence, there exists 
$\phi_D:D_1\to D_2$ and $a\in\ZZ_p$ such that $\Tilde\phi_D=a\M_\QQ(\phi_D)$. 
This completes the proof.
\end{proof}

\begin{prop}\label{T_localization}
Assume that there exists a quasi-free contravariant \Die module $M_{\T}$.
 Then the functor 
\[\F_{\T}:\RM{\T^\op}\to \RM{\T^\op\locp}\times_{\Md{D_{S}^\op\locp}}\TF{D_{S}^\op},\]
where $D\mapsto (D, \M(D))$, is an equivalence of categories.
\end{prop}
\begin{proof}
    First, we show that the functor is full and faithful. According to Proposition~\ref{full_faithful}, it suffices to show that for any pair of $p$-rigid left
    $\T^\op$-modules $D_1$ and $D_2$ the natural morphism
    \[\Hom_{\T^\op}(D_1,D_2)\otimes\ZZ_p\to \Hom_{D_S^\op}(\M(D_1),\M(D_2))\]
    is an isomorphism. Indeed, since $D_1$ and $D_2$ are $p$-rigid and $M_{\T}$ is quasi-free, the morphism 
    \[\Hom_{D_S^\op}(\M(D_1),\M(D_2))\to 
    \Hom_{\T_p^\op}(D_p(\M(D_1)),D_p(\M(D_2)))\cong  \Hom_{\T^\op}(D_1,D_2)\otimes_\ZZ\ZZ_p \] is the inverse morphism.        
    
     Now, we prove essential surjectivity.
    Let $D'$ be a left $\T^\op\locp$-module, and let $M$ be a $D_S^\op$-module with an isomorphism $\M_\QQ(D')\cong M\locp$. Choose a left $\T^\op$-module $D_0$ such that
    $D_0\locp\cong D'$. The morphism
\[M\to M\locp\cong \M_\QQ(D')\cong \M(D_0)\locp\]
    corresponds to a morphism
    \[D_p(M)\to D_p(\M(D_0))\locp\cong D_0\otimes_{\ZZ}\QQ_p\cong D'\otimes_{\ZZ\locp}\QQ_p.\]
    This morphism is injective since $D$ is locally free over $\OO_L\locp$.
        Put $D=D'\times_{D'\otimes\QQ_p}D_p(M).$
        Then $\M(D)\cong\hM(D_p(M))\cong M,$ and $\F_\T(D)\cong(D',M)$. 
        \end{proof}

\subsection{An equivalence theorem.}
Let $A_{\T}$ be a wide abelian variety over $\T$. 
Define a covariant functor from $\AV_{S}$ to $\Md{\T^\op}$ as follows:
\[\D_\T:A\mapsto\Hom_{\ZZ}(\Hom(A,A_\T),\ZZ).\]
where the action of $\T^\op$ comes from the action of $\T$ on $A_\T$.
Let $\D_{\T}^\circ=\D_\T\otimes\QQ.$

\begin{lemma}\label{simple}
   The module $\D^\circ_\T(B_\pi)$ is a simple $\HH_\pi^\op$-module.
   For any abelian variety $A$ from $\AV_{S}$, the module $\Phi_\T(A)$ is a locally free $\OO_L$-module of rank $2\dim A$. 
\end{lemma} 
\begin{proof}
Recall that $A_\T$ is isogenous to $\prod_{\pi\in\Pi_0}B_\pi^{d_{\pi}}$,
where $d_\pi=m/e_\pi$; therefore, the dimension of the $\QQ$-vector space 
$\Hom^\circ(B_\pi,A_\T)$ dual to $\D^\circ_\T(B_\pi)$ 
is equal to \[d_\pi\dim_\QQ\End^\circ(B_\pi)=2m\dim B_\pi.\]
Since $\HH_{\pi}^\op\cong\Mat_{d_\pi}\End^\circ(B_\pi)^\op$ is the matrix algebra over the
skew-field $\End^\circ(B_\pi)^\op$, the module $\Phi_\T^\circ(B_\pi)$ is simple over this algebra.

 The second part of the lemma follows from the observation that any abelian variety $A$ from $\AV_{S}$ is isogenous to a product of $B_\pi$.
\end{proof}

\begin{prop}\label{rigid_prop}
Let $\ell\neq p$ be a prime number, and let $A$ be an abelian variety from $\AV_S$.
Assume that there exists a wide abelian variety $A_\T$ over $\T$. 
Consider the functors $\cT_\ell, \hT_\ell$ and $D_\ell$ associated with the tame structure 
$\{T_\ell^*(A_\T)\}_{\ell\neq p}$.
\begin{enumerate}
\item  Then $\D_\T(A)$ is $\ell$-rigid and there is a natural isomorphism of left $\T^\op$-modules
\[\D_\T(A)\otimes_\ZZ\ZZ_\ell\cong D_\ell(T_{\ell}(A)).\]
\item Assume that $M_{\T}=M^*(A_\T)$ is quasi-free.
Then $\D_\T(A)$ is $p$-rigid and  there is a natural isomorphism of left $\T^\op$-modules:
\[\D_\T(A)\otimes_\ZZ\ZZ_p\cong D_p(M(A)).\]
\end{enumerate}
    \end{prop}
\begin{proof}
If $\ell\neq p$, then the adjunction isomorphism
\[\Hom_{\ZZ_\ell}(T_\ell(A)\otimes_{S} T_{\ell,\T},\ZZ_\ell) 
\to\Hom_{S}(T_\ell(A),\Hom_{\ZZ}(T_{\ell,\T},\ZZ_\ell)) \]
induce an isomorphism 
\[T_\ell(A)\otimes_{S} T_{\ell,\T}\to 
\Hom_{\ZZ_\ell}(\Hom_{S}(T_\ell(A),T_\ell(A_\T)),\ZZ_\ell).\]
According to Tate's theorem,
\[T_\ell(A)\otimes_{S} T_{\ell,\T}\cong \D_\T(A)\otimes_\ZZ\ZZ_\ell,\]
and, according to Proposition~\ref{Tate_eq}.(1),
$\hT_\ell(D_\ell(T_\ell(A)))\cong T_\ell(A);$
therefore, $\D_\T(A)$ is $\ell$-rigid:
\[D_\ell(T_\ell(\D_\T(A)))\cong D_\ell(\hT_\ell(D_\ell(T_\ell(A))))\cong D_\ell(T_\ell(A))\cong \D_\T(A)\otimes_\ZZ\ZZ_\ell.\]
Part $(1)$ is proved.
In the same way, for \Die modules there is an isomorphism
\[\Hom_{\ZZ_p}( M_{\T}\otimes_{D_S^\op}M(A),\ZZ_p) 
\to\Hom_{D_S^\op}(M(A),\Hom_{\ZZ_p}(M_{\T},\ZZ_p)). \]
According to Theorem~\ref{WM_thm},
\[M_{\T}\otimes_{D_S^\op}M(A) \cong 
\Hom_{\ZZ_p}(\Hom_{S}(M(A),M(A_\T)),\ZZ_p)\cong \D_\T(A)\otimes_\ZZ\ZZ_p.\]
It follows from the definition of a quasi-free \Die module, that
\[D_p(\M(\D_\T(A))))\cong D_p(\hM(D_p(M(A))))\cong D_p(M(A))\cong \D_\T(A)\otimes_\ZZ\ZZ_p;\]
therefore, $\D_\T(A)$ is $p$-rigid. Part $(2)$ is proved.
\end{proof}

We proved that if $M^*(A_\T)$ is quasi-free, then the functor $\D_\T$ is a composition of a functor
\[\Phi_\T:\AV_{S}\to\RM{\T^\op}\]
with the natural forgetful functor to $\TF{\T^\op}$.
In general, there is a functor 
\[\Phi'_\T:\AV_{S}\to \RM{\T^\op\locp}\times_{\Md{D_{S}^\op\locp}}\TF{D_{S}^\op},\]
where $A\mapsto (\Phi_\T(A), M(A))$.

\begin{thm}\label{main_thm_cyclic} 
If there exists a wide abelian variety $A_\T$, then the functor
$\D_\T\locp$ is full and faithful, and $\Phi'_\T$ is an equivalence of categories.
\end{thm}
\begin{proof}
Firstly, we prove that $\D_\T\locp$ is full and faithful.
 It is enough to prove that for any prime $\ell\neq p$ the natural morphism
 \begin{equation}\label{eq_Hom}
     \Hom(A,B)\otimes_\ZZ\ZZ_\ell\to \Hom_{\T^\op} (\D_\T(A),\D_\T(B))\otimes_\ZZ\ZZ_\ell
 \end{equation}  
        is an isomorphism. 
 If $\ell\neq p$, we consider the functors $\cT_\ell, \hT_\ell$ and $D_\ell$ associated with the tame structure $\{T_\ell^*(A_\T)\}_{\ell\neq p}$.
According to Tate's Theorem, Proposition~\ref{Tate_eq}.(1), and Proposition~\ref{rigid_prop}.(1), the morphism (\ref{eq_Hom}) is a composition of the natural isomorphisms
\[\Hom(A,B)\otimes_\ZZ\ZZ_\ell\cong\Hom_{S_\ell}(T_\ell(A),T_\ell(B))\cong\]
 \[\cong\Hom_{S_\ell}(\hT_\ell(D_\ell(T_\ell(A))),\hT_\ell(D_\ell(T_\ell(B))))\cong\]
 \[\cong\Hom_{\T^\op}(D_\ell(T_\ell(A)), D_\ell(T_\ell(B))))\cong\]
 \[\cong\Hom_{\T^\op}(\D_\T(A),\D_\T(B))\otimes_\ZZ\ZZ_\ell.\]

It follows that $\Phi_\T\locp$ is full and faithful as well.
Now, we prove that $\Phi'_\T$ is essentially surjective. 
Let $D$ be a rigid $\T^\op\locp$-module, and let
$M$ be a $D_S^\op$-module with an isomorphism $\M_\QQ(D)\cong M\locp$. 
According to Lemma~\ref{simple}, $D\otimes_\ZZ\QQ\cong\D^\circ_\T(B)$, 
where $B=\prod_{\pi\in\Pi_0}B_\pi^{m_\pi}$ for some $m_\pi\in\NN$.
Without loss of generality, we may assume that there is an injective morphism 
$t_D:D\to\Phi_\T(B)\locp$ such that $M$ goes to a submodule of $M(B)$ under the morphism
\[t_{p,D}:M\to \M_\QQ(D)\to\M_\QQ(\Phi_\T(B)\locp)\cong M(B)\locp\]
induced by $\M_\QQ(t_D)$.

Clearly, for almost all primes $\ell$ the morphism $t_D\otimes_\ZZ\ZZ_\ell$ is an isomorphism. 
For any $\ell\neq p$ such that $t_D\otimes\ZZ_\ell$ is not an isomorphism, we get an injective morphism \[t_{\ell,D}:\cT_\ell(D)\to \cT_\ell(\Phi_\T(B))\cong T_\ell(B).\] 
By Lemma~\ref{iso_lemma} there exists an isogeny $\phi':A'\to B$ such that 
$T_\ell(\phi')(T_\ell(A'))=t_{\ell,D}(\cT_\ell(D))$. 
Since $D$ is $\ell$-rigid, $\phi'$ induces an isomorphism 
\[D\otimes\ZZ_\ell\cong D_\ell(\cT_\ell(D))\cong D_\ell(T_\ell(A')) \cong\Phi_\T(A')\otimes_\ZZ\ZZ_\ell.\]

The morphism $t_{p,D}$ induces a morphism $M\to M(B)\cong M(A').$
As before, there is an isogeny $\phi: A\to A'$ such that the corresponding morphism of 
$D_S^\op$-modules induces an isomorphism $M(A)\cong M$.  
We finally get an abelian variety $A$ with an isogeny $\phi:A\to B$ such that the morphism 
\[\Phi_\T(\phi)\locp:\Phi_\T(A)\locp\to \Phi_\T(B)\locp\]
induces a morphism $\Phi_\T(A)\locp\to D$ that is an isomorphism locally; therefore, $\Phi_\T(A)\locp\cong D$. It follows that $\Phi'_\T(A)\cong(D,M)$.
\end{proof}

\begin{corollary}\label{main_thm}
If there exists a wide abelian variety $A_{\T}$ over $\T$ such that $M^*(A_\T)$ is quasi-free,
then $\D_\T$ is full and faithful and $\Phi_\T$ is an equivalence of categories.
\end{corollary}
\begin{proof}
As before, to prove that $\D_\T$ is full and faithful, it is sufficient to show that for any prime $\ell$ the natural morphism
 \begin{equation}
     \Hom(A,B)\otimes_\ZZ\ZZ_\ell\to \Hom_{\T^\op}(\D_\T(A),\D_\T(B))\otimes_\ZZ\ZZ_\ell
     \end{equation}
  is an isomorphism. If $\ell\neq p$ it follows from Theorem~\ref{main_thm_cyclic}. 
  According to Proposition~\ref{rigid_prop}.(2) and the definition of a quasi-free \Die module we have:
\[\Hom(A,B)\otimes_\ZZ\ZZ_p\cong\Hom_{D_S^\op}(M(A),M(B))\cong\]
 \[\cong\Hom_{D_S^\op}(\hM(D_p(M(A))),\hM(D_p(M(B))))\cong\]
 \[\cong\Hom_{\T^\op}(D_p(M(A)), D_p(M(B))))\cong\]
\[\cong\Hom_{\T^\op}(\D_\T(A),\D_\T(B))\otimes_\ZZ\ZZ_p.\]

We now show that $\Phi_\T$ is an equivalence.
 The functor $\Phi'_\T$ is a composition of $\Phi_\T$ with an equivalence $\F_\T$ of Proposition~\ref{T_localization}.
According to Theorem~\ref{main_thm_cyclic}, $\Phi'_\T$ is an equivalence; therefore,
the full and faithful functor $\Phi_\T$ is an equivalence as well.
 \end{proof}

\begin{proof}[Proof of Theorem~\ref{main_simple}]
The family $\{T_{\ell,\T,\Pi}\}_{\ell\neq p}$ is a balanced tame structure.
By Proposition~\ref{wide_AV}, there exists a wide abelian variety $A_\T$ endowed with the morphism $h_{A_\T}:\T\to\End(A_\T)$ such that $h_{A_\T}\locp$ is an isomorphism. The theorem now follows from Theorem~\ref{main_thm_cyclic}, Proposition~\ref{Tate_eq}, and Remark~\ref{die_dual}.
\end{proof}

\subsection{Tate modules over the ramified prime $s$.}\label{s_section}
In this section, $L/\QQ$ is a cyclic field extension of degree $m$ fully and tamely ramified at a prime $s$. 
We assume that $m$ divides $s-1$, and $S$ is an $s$-maximal order in $E_\Pi$. 

Let $\T=\T(\lambda)\subset\HH_{\Pi,L}(\lambda)$
be an $s$-cyclic and $s$-tame order, where $\lambda\in S$ is an $s$-unit. 

\begin{rem}\label{local_norm_rem}
According to Lemma~\ref{tame_module} there exists
$\omega'_s\in (S_s\otimes_\ZZ\OO_L)^*$ such that $N_L(\omega'_s)=\lambda$.
Since the ramification is tame and full, $s\OO_L=\p^m$, where $\p$ is the only prime ideal over $s$. Hence, \[\lambda\equiv (\omega'_s)^m\pmod{S_s\otimes_\ZZ\p }.\] 
By the Hensel Lemma, there exists $\omega_s\in S_s^*$ such that $\omega_s^m=\lambda$.
\end{rem}

There is a natural isomorphism \[\T_s\to S_s\otimes_\ZZ\OO_L[\Gal(L/\QQ)]\] given by
$G_\lambda\mapsto \omega_s(1\otimes [g])$,
where $\Gal(L/\QQ)\cong\ZZ/m\ZZ$ is generated by $g$.
%, where $(\OO_L/\ZZ,g,1)$ is the cyclic ring defined in Section~\ref{cyclic_ring}.

Since $m$ divides $s-1$, we have $\zeta=\zeta_m\in\ZZ_s$, and there is an isomorphism
\[\ZZ_s[\Gal(L/\QQ)]\cong\prod_{i=0}^{m-1}v_i\ZZ_s,\]
where 
\[v_i=\frac{1}{m}\sum_{j=0}^{m-1}\zeta^{ij}[g]^j.\]

According to Kummer theory, $L_s=\QQ_s(\beta_s)$, where $\beta_s^m=u_ss$, $u_s\in\ZZ_s^*$, and $g(\beta_s)=\zeta^{-1}\beta_s$.
The group algebra $\ZZ_s[\Gal(L/\QQ)]$ acts on $\ZZ_s[\beta_s]$ as follows:
\[v_i\beta_s^a=\left\{\begin{array}{ll}
    \beta_s^a, & \hbox{if } i=a;\\
    0, & \hbox{if } i\neq a.\\
    \end{array} \right .\]
Choose $\beta\in\OO_L$ such that $\beta \equiv \beta_s \pmod{s}$.
Then $\beta^m=us$, where $u$ is invertible in $\OO_L\otimes_\ZZ\ZZ_s$.
Moreover, according to~\cite[Theorem 1 on p. 23]{CF}, $\OO_L\otimes\ZZ_s\cong\ZZ_s[\beta]$. 
Choose $\alpha\in S$ such that \[\omega_s-\alpha\in sS_s.\]

\begin{deff}\label{H_operator}
Put $H_s=\frac{\beta_s^{m-1}}{s}(v_0-1)\in L_s[\Gal(L/\QQ)]$, and
\[H=\frac{\beta^{m-1}}{s}(\sum_{i=1}^{m-1}\alpha^{m-i}G_\lambda^i-(m-1)\alpha^m)\in\HH_{\Pi,L}(\lambda).\]  \end{deff}

By construction, $H-m\lambda H_s\in\T_s$. Moreover, 
\[H_s\beta_s^a=\left\{\begin{array}{ll}
    -u\beta_s^{a-1}, & \hbox{if } a>0;\\
    0, & \hbox{if } a=0.\\
    \end{array} \right .\]
In particular, $H_s$ is nilpotent.

\begin{lemma}
Let $\widetilde{\T}_s^\max\subset L_s[\Gal(L/\QQ)]$ be the order generated by 
$\OO_L[\Gal(L/\QQ)]_s$ and $H_s$. 
The action of $\widetilde{\T}_s^\max$ on \[\OO_L\otimes_\ZZ\ZZ_s\cong\ZZ_s[\beta_s]\]
induces an isomorphism $\widetilde{\T}_s^\max\cong\Mat_m(\ZZ_s).$   
\end{lemma}
\begin{proof}
     Let $I_{ij}$ denote the elementary $m\times m$-matrix.
    The elements $v_i$ act as diagonal elementary matrices $I_{ii}$.
    If $1\leq i\leq m-1$, then 
    \[\beta_s v_{i-1}\in \OO_{L,s}^*I_{i(i-1)},\text{ and } H_s v_i\in \OO_{L,s}^*I_{(i-1)i}.\] 
     Clearly, these matrices generate $\Mat_m(\ZZ_s)$.
\end{proof}

\begin{corollary}\label{s_max_order}
    The order $\T^\max$ generated by $\T$ and $H$ is maximal at $s$ and is $s$-balanced.
    If a left $\T^\op$-module $D$ is $H$-invariant, then $D$ is $s$-rigid.
\end{corollary}

\subsection{Supersingular abelian varieties over $\FF_q$.}
Let $\Pi$ be a set of supersingular Weil numbers.  
\begin{thm}\label{ss_thm}
 There exists a prime $s\not\in\{2,p\}\cup\L_\Pi$
 such that in the quadratic extension $L=\QQ(\sqrt{-s})$ any $\ell\in\L_\Pi$  splits completely and $p$ is inert.
\end{thm}
\begin{proof}
Let $\L_r=\{\ell\in\L_\Pi|\ell\equiv r\pmod 4\}$, where $r\in\{1,3\}$.
Denote by $\Delta_r$ the subgroup of $\QQ^*$ generated by $\L_r$.
Assume that $p>2$. The Galois group of $\QQ(\zeta_8, p^{1/2})$ over $\QQ$ is isomorphic to $(\ZZ/2\ZZ)^3$.
According to the Chebotarev density theorem, there exist infinitely many primes
$s\equiv -1\pmod 8$ that split completely in 
\[\QQ(\Delta_1^{1/2}, \Delta_3^{1/2}),\]
and inert in $\QQ(p^{1/2})$. In this case, $2$ is split and $p$ is inert in $\QQ(\sqrt{-s})$, because
 \[\leg{-s}{p}=\leg{-1}{p}\leg{s}{p}=
 (-1)^{\frac{p-1}{2}}(-1)^{\frac{p-1}{2}\frac{s-1}{2}}\leg{p}{s}=\leg{p}{s}=-1\]
Assume that $p=2$. In this case, $\sqrt{2}\in\QQ(\zeta_8)$, and 
\[\Gal(\QQ(\zeta_8)/\QQ)\cong(\ZZ/2\ZZ)^2.\]
According to the Chebotarev density theorem, there exist infinitely many primes 
$s\equiv 3\pmod 8$ that split completely in $\QQ(\Delta_1^{1/2}, \Delta_3^{1/2})$. 
Clearly, $2$ is inert in $\QQ(\sqrt{-s})$.

In both cases, if $\ell\in\L_1$, then
 \[\leg{-s}{\ell}=\leg{-1}{\ell}\leg{s}{\ell}=(-1)^{\frac{\ell-1}{2}}\leg{\ell}{s}=1.\] 
On the other hand, if  $\ell\in\L_3$, then
   \[\leg{-s}{\ell}=\leg{-1}{\ell}\leg{s}{\ell}=
   (-1)^{\frac{\ell-1}{2}}(-1)^{\frac{\ell-1}{2}\frac{s-1}{2}}\leg{\ell}{s}=\leg{\ell}{s}=1.\] 
Hence, any $\ell\in\L_\Pi$ splits in $\QQ(\sqrt{-s})$.  
\end{proof}

Fix $L=\QQ(\sqrt{-s})$ such that $p$ is inert and any $\ell\in\L_\Pi$ splits completely in $L$. Put $\HH_{\Pi,L}=(L/\QQ,g,-p)\otimes_\QQ E_\Pi$, where $g\in\Gal(L/\QQ)$ is the non-trivial element.
Let \[H=\beta(\alpha-G_\lambda)/s\in\HH_{\Pi,L}\] be the endomorphism of Definition~\ref{H_operator}, and let $\T=\T_{S,L}(-p,p)$ be the order of Example~\ref{L_order}. 

\begin{lemma}\label{SS_Dmod}
    Assume that $n=2$ and $F_\Pi=\eps_p p$, where $\eps_p\in\{\pm 1\}$.
    Then there is an isomorphism $\T_p^\op\cong D_S$.
\end{lemma}
\begin{proof}
    There exists $\mu\in W(\FF_{p^2})^*$ such that 
    $\sigma(\mu)\mu=-\eps_p$. The isomorphism $\OO_L\otimes\ZZ_p\cong W(\FF_{p^2})$ extends to an isomorphism $\T_p^\op\cong D_S$ as follows: $G\mapsto \mu F$.
\end{proof}

\begin{thm}\label{ss_thm_main}
Let $\T^\max$ be the order generated by $H$ and $\T_{S,L}(-p,p)$.
\begin{enumerate}
    \item There exists an abelian variety $A_{\Pi,L}$ from $\AV_\Pi$ such that
    $\HH_{\Pi,L}\cong\End^\circ{A_{\Pi,L}}$.
    \item The functor $\D'_{\T^\max}$ is an equivalence of categories.
\item If $n=2$ and $F_\Pi=\pm p$, then $\D_{\T^\max}$ is an equivalence of categories.
\end{enumerate}
\end{thm}
\begin{proof}
Note that for any $\pi\in\Pi$ the number $e_\pi$ divides $2$.
Let \[A_{\Pi,L}=\prod_{\pi\in\Pi_0}B_\pi^{2/e_\pi}.\]
We claim that for any $\pi\in\Pi_0$ there is an isomorphism
\[\End^\circ(B_\pi^{2/e_\pi})\cong \HH_{\pi,L}=(L/\QQ,g,-p)\otimes_\QQ \QQ(\pi).\] 
Indeed, both algebras are central simple over $\QQ(\pi)$ of the same dimension; therefore, it is enough to prove that their local invariants are equal.
If $\ell\not\in \{p,s\}$, then the invariants of both algebras are trivial.
On the other hand, 
\[\leg{-p}{s}=\leg{-1}{s}\leg{p}{s}=-(-1)^{\frac{s-1}{2}}=1;\]
therefore, $-p=x^2$ for some $x\in\QQ_s$. According to Theorem~\ref{Th30_4}, the invariant of $\HH_{\pi,L}$ is also trivial.
Since $\pi$ is supersingular, for any place $v$ of $\QQ(\pi)$ over $p$ we have
\[\inv_p(L/\QQ,g,-p)=\frac{v(\pi)}{v(q)}=\frac{1}{2}.\]
According to~\cite[Theorem 31.9]{MO},
\[\inv_v(\HH_{\pi,L})=\frac{v(\pi)}{v(q)}[\QQ(\pi)_v:\QQ_p]=\inv_v(\End^\circ(B_\pi)).\]
Finally, if $\pi\in\RR$, then both invariants at $\infty$ are non-trivial, 
because $-p<0$, and $\QQ(\pi)\subset\RR$.
If $\pi\not\in\RR$, then both invariants at $\infty$ are trivial, because 
$\QQ(\pi)$ is a CM-field. Part $(1)$ is proved.

The order $\T^\max$ has a balanced tame structure by Proposition~\ref{balanced_order} and Corollary~\ref{s_max_order}. Now, part $(2)$ follows from Theorem~\ref{main_simple}.

If $n=2$, then $\T^\max_p\cong D_S$ by Lemma~\ref{SS_Dmod}.
Hence, $\T^\max$ is $p$-balanced, and by Proposition~\ref{p_rigid_prop}
and Corollary~\ref{main_thm}, the functor $\D_{\T^\max}$ is an equivalence.
\end{proof}

\begin{proof}[Proof of Theorem~\ref{ss_main}]
According to Theorem~\ref{ss_thm}, there exists $s\equiv 3\pmod 4$ such that
in the extension $L=\QQ(\sqrt{-s})$ the prime $p$ is inert and any $\ell\in\L_\Pi$ splits completely.
Note that supersingular Deligne modules are modules over $\T_{R_\Pi,L}(-p,p)$.
Apply Theorem~\ref{ss_thm_main}. 
\end{proof}

\section{Admissible field extensions}\label{sec5}
This section is devoted to $\Pi$-admissible and $\Omega$-admissible extensions.

\subsection{Kummer theory.}
Let $K$ be a field, and let $m$ be a natural number. Assume that $K$ contains a primitive $m$-th root of unity $\zeta_m$. To any finitely generated subgroup $\Delta$ of $K^*$ we 
associate a Kummer extension $K(\Delta^{1/m})/K$.
There is an isomorphism of finite abelian groups
\[\Gal(K(\Delta^{1/m})/K)\cong\Hom({\Delta(K^*)^m}/{(K^*)^m},\mu_m),\]
where $\sigma\in \Gal(K(\Delta^{1/m})/K)$ goes to a homomorphism
$a\mapsto \sigma(a^{1/m})/a^{1/m}$.

\begin{lemma}\label{Kummer_lemma2}
    Let $K=K_0(\Delta^{1/m})$ be a Kummer extension, and let $x\in K_0$.
    If $[K(x^{1/m}):K]=m/d$ for some divisor $d$ of $m$, then
    $x^{m/d}\in \Delta(K_0^*)^m$. 
    \end{lemma}
\begin{proof}
Let $\Delta_x\subset K_0^*$ be generated by $\Delta$ and $x$.
Apply Kummer theory to the extensions $K_0(\Delta^{1/m})/K_0$ and     
$K_0(\Delta_x^{1/m})/K_0$. We get a homomorphism
\[\alpha:{\Delta(K_0^*)^m}/{(K_0^*)^m}\to {\Delta_x(K_0^*)^m}/{(K_0^*)^m}\]
such that \[\Hom(\coker\alpha,\mu_m)\cong\Gal(K(x^{1/m})/K).\]
Since $m/d$ annihilates $\coker\alpha$, we have $x^{m/d}\in \Delta(K_0^*)^m$.
\end{proof}

\noindent

\begin{lemma}\label{Kummer_lemma}
    Let $s\neq \ell$ be the prime numbers. Let $m$ be a divisor of $s-1$, and let $L\subset\QQ(\zeta_s)$ be the subfield of degree $m$. Then the local degree $[\Hat{L}_\ell:\QQ_\ell]$ is equal to the order of   $\ell^\frac{s-1}{m}$ in $\FF_s^*$.
     \end{lemma}
\begin{proof}
Since $\ell\neq s$, the extension $L/\QQ$ is unramified at $\ell$.
    The $\ell$-th Frobenius automorphism of $\QQ(\zeta_s)$ is given by $\zeta_s\mapsto \zeta_s^\ell$.
The prime $\ell$ generates a finite subgroup $\Delta_\ell$ of order $d$ in the Galois group
\[\FF_s^*\cong\Gal(\QQ(\zeta_s)/\QQ).\] 
The stabilizer of any prime ideal over $\ell$ is generated by the $\ell$-th Frobenius,
hence, \[\Delta_\ell\cong\Gal(\QQ_\ell(\zeta_s)/\QQ_\ell).\]  
 We have a surjective morphism of cyclic Galois groups 
\[\Gal(\QQ(\zeta_s)/\QQ)\to\Gal(L/\QQ),\] 
and the order of the kernel is equal to $\frac{s-1}{m}$. 
Its restriction to $\Delta_\ell$ induces the surjective morphism
\[\Delta_\ell\to\Gal(\hat{L}_\ell/\QQ_\ell).\]
Clearly, the order of the image is equal to the order of $ \ell^\frac{s-1}{m}$ in $\FF_s^*$.
 \end{proof}

If $r$ is the rank of $\Delta$, then the degree of Kummer extension $K(\zeta_{am},\Delta^{1/m})$ grows linearly with respect to $m^r$.

\noindent

\begin{thm}~\cite[Theorem 13]{PS}\label{Kummer}
Let $K$ be a number field and let $\Delta$ be a finitely generated torsion-free subgroup of $K^*$ of rank $r>0$. There exists an integer $C$ such that for all natural $m$ and $a$ the ratio
\[ \frac{m^r}{[K(\zeta_{am},\Delta^{1/m}):K(\zeta_{am})]}\] is an integer and divides $C$.
\end{thm}

From this theorem we derive the main technical result of this section.

\begin{thm}\label{Kummer0}
Let $K$ be a number field and let $\Delta$ be a finitely generated torsion-free subgroup of $K^*$ of rank $r$.
Then there exist natural numbers $M_0$ and $C$ such that 
if $M_0$ divides a natural number $m$, then for any $a\in\NN$ we have
\begin{enumerate}
    \item $[K(\zeta_{am},\Delta^{1/m}):K(\zeta_{am})]=m^r/C,$
    \item $[K(\zeta_{am},\Delta^{1/am}):K(\zeta_{am},\Delta^{1/m})]=a^r,$  and
    \item $[K(\zeta_{am},\Delta^{1/m}):K(\zeta_{m},\Delta^{1/m})]=[K(\zeta_{am}):K(\zeta_{m})].$
\end{enumerate}
    \end{thm}
\begin{proof}
If $r=0$, then the statement is trivial. Assume that $r>0$.
According to Theorem~\ref{Kummer}, the ratio 
\[C(a,m)=\frac{m^r}{[K(\zeta_{am},\Delta^{1/m}):K(\zeta_{am})]}\] is integral and bounded.
Fix natural numbers $a_1$ and $m_1$.
We claim that for any $m$ and $a$ such that $a_1m_1$ divides $m$ the ratio
$C(a_1,m_1)$ divides $C(a,m)$.
Indeed,
\[\frac{m^r}{[K(\zeta_{am},\Delta^{1/m}):K(\zeta_{am})]}=
\frac{m^r[K(\zeta_{am}):K(\zeta_{a_1m_1})]}{[K(\zeta_{am},\Delta^{1/m}):K(\zeta_{am})][K(\zeta_{am},):K(\zeta_{a_1m_1})]}=\]
\[=\frac{(m/m_1)^r}{[K(\zeta_{am},\Delta^{1/m}):K(\zeta_{am},\Delta^{1/m_1})]}
\frac{[K(\zeta_{am}):K(\zeta_{a_1m_1})]}{[K(\zeta_{am},\Delta^{1/m_1}):K(\zeta_{a_1m_1},\Delta^{1/m_1})]}
\frac{m_1^r}{[K(\zeta_{a_1m_1},\Delta^{1/m_1}):K(\zeta_{a_1m_1})]}.\]
The group $\Delta^{1/m_1}$ is a product of a free group $\Delta_{m_1}$ and the group of roots of unity $\mu_{m_1}$; therefore,
\[\frac{(m/m_1)^r}{[K(\zeta_{am},\Delta^{1/m}):K(\zeta_{am},\Delta^{1/m_1})]}=
\frac{(m/m_1)^r}{[K(\zeta_{am},\Delta_{m_1}^{m_1/m}):K(\zeta_{am},\Delta_{m_1})]}
\]
is an integer. On the other hand, the number $[K(\zeta_{am},\Delta^{1/m_1}):K(\zeta_{a_1m_1},\Delta^{1/m_1})]$ divides 
$[K(\zeta_{am}):K(\zeta_{a_1m_1})]$. The claim is proved. It follows that if
\[C=\frac{m_1^r}{[K(\zeta_{a_1m_1},\Delta^{1/m_1}):K(\zeta_{a_1m_1})]}\]
is the maximum of the ratio for all natural $m$ and $a$, then
for all $m$ and $a$ such that $a_1m_1$ divides $m$ the product of natural numbers
\[\frac{m^r}{[K(\zeta_{am},\Delta^{1/m}):K(\zeta_{am})]}=
C\cdot\frac{(m/m_1)^r}{[K(\zeta_{am},\Delta^{1/m}):K(\zeta_{am},\Delta^{1/m_1})]}\cdot
\frac{[K(\zeta_{am}):K(\zeta_{a_1m_1})]}{[K(\zeta_{am},\Delta^{1/m_1}):K(\zeta_{a_1m_1},\Delta^{1/m_1})]}.\]
divides $C$. This proves $(1)$. Parts $(2)$ and $(3)$ are now easy.
Put $M_0=a_1m_1$. The theorem is proved.
\end{proof}

\subsection{Admissible primes.}
Before we state the main technical result of this section, we need some notation.
Let $\Pi$ be a $\Gal(\BQQ/\QQ)$-invariant set of algebraic numbers $\pi$ such that $\pi\Bar\pi=q^{w(\pi)}$, where $q=p^n$ and $w(\pi)\in\ZZ$,
and let $\L$ be a finite set of prime numbers such that $p\not\in\L$. 

\begin{deff}
Let $m=nn_0m_0$ for some natural numbers $m_0$ and $n_0$. 
Choose $\eps\in\{\pm 1\}$.
We say that a prime number $s\in\ZZ$ is \emph{$\Omega$-admissible}, where 
 $\Omega=(n,m,m_0,\Pi,\L,\eps)$, if the following conditions hold:
\begin{enumerate}
\item $m$ divides $s-1$;
\item $nm_0$ is equal to the order of $p^\frac{s-1}{m}$ in $\FF_s^*$;
%\item $(\frac{s-1}{m},m)=1$;
\item for any $\ell\in\L$ we have $\ell^\frac{s-1}{m}\equiv 1\pmod s$;
%\item if $wn$ is odd and $\Pi\cap\RR\neq\emptyset$, then $2|m_0$;
\item if $\Pi\cap\RR\neq\emptyset$, then $\frac{s-1}{m}$ is odd, and $nm_0$ is even;
%$-1$ is not a $m$-th power in $\QQ_s$;
\item for any $\pi\in\Pi$ the product $\eps\pi^{m_0}$ is an $m$-th power in $\QQ_s(\pi)$.
\end{enumerate}
\end{deff}

\begin{deff}
Fix $(n,m_0,\Pi,\L)$. Let $\Delta_{m_0}\subset\overline{\QQ}^*$ be the group generated by $\L$ and all the quotients of the form $(\pi/\pi')^{m_0}$, where $\pi,\pi'\in\Pi$.
Put $K=\QQ(\Delta_{m_0})$.
\end{deff}

\begin{rem}
    Clearly, for any $\Pi$ there exists $m_0$ such that the group $\Delta_{m_0}$ is torsion-free. 
\end{rem}

Fix $(n,m,m_0,\Pi,\L)$. Put $K_{m,m_0}=K(\zeta_m,(\Delta_{m_0})^{1/m}).$

\begin{prop}\label{M_prop}
Fix $(n,m_0,\Pi,\L)$ such that $\Delta_{m_0}$ is torsion-free.
Then there exists a natural number $N_0$ such that
if $nm_0N_0$ divides a natural number $m$, then 
\begin{itemize}
\item[($a$)] $nm_0$ divides $[K_{m,m_0}(p^\frac{1}{m}):K_{m,m_0}]$;
\item[($b$)] if $m$ is even, then 
$[K_{m,m_0}(\zeta_{2m},p^\frac{1}{m}):K_{m,m_0}(p^\frac{1}{m})]=2$.
\end{itemize}
\end{prop}
\begin{proof}
Let $\Delta_+\subset K$ be the group generated by $\Delta_{m_0}$ and $p$,
and let $r$ be the rank of $\Delta_{m_0}$.
We claim that $\Delta_+$ is a torsion-free group of rank $r+1$.
Indeed, note that if $\delta\in\Delta_{m_0}$, then $\delta\Bar\delta\in\QQ^*$ is supported in $\L$. If there exists a relation of the form $p^a\delta=1$, then $p^{2a}\delta\Bar\delta=1$; hence, $a=0$, and the relation is trivial.

According to Theorem~\ref{Kummer0}.(1), there exist 
 $C_1$, $C_2$, and $M_0$ such that if $M_0$ divides $m$, then 
 \[[K_{m,m_0}:K(\zeta_m)]=\frac{m^r}{C_1}\text{, and }\]
 \[[K_{m,m_0}(p^\frac{1}{m}):K(\zeta_m)]=\frac{m^{r+1}}{C_2}.\]
 It follows that,\[[K_{m,m_0}(p^\frac{1}{m}):K_{m,m_0}]=\frac{mC_1}{C_2}.\]
 Therefore, if $nm_0C_2$ divides $mC_1$, then
 $nm_0$ divides $[K_{m,m_0}(p^\frac{1}{m}):K_{m,m_0}]$.

 Finally, apply Theorem~\ref{Kummer0}.(3) to $\Delta_+$ with $a=2$.
 There exists a minimal number $N_{1}$ such if $m_0N_{1}$ divides $m$, then
 \[[K_{m,m_0}(\zeta_{2m},p^\frac{1}{m}):K_{m,m_0}(p^\frac{1}{m})]=2.\]
 The proposition is proved.
\end{proof}

\begin{rem}
    This result can be improved as follows. Assume that the $2$-torsion of $\Delta_{m_0}$ is trivial. Then there exists a natural number $N_2$ such that
if $nm_02^{N_2}$ divides a natural number $m$, then 
\begin{itemize}
\item[($a$)] $nm_0$ divides $[K_{m,m_0}(p^\frac{1}{m}):K_{m,m_0}]$;
\item[($b$)] if $m$ is even, then 
$[K_{m,m_0}(\zeta_{2m},p^\frac{1}{m}):K_{m,m_0}(p^\frac{1}{m})]=2$.
\end{itemize}
In particular, there exist natural $m_0$ and $m$ such that $m/n$ is a power of $2$ and conditions $(a)$ and $(b)$ hold.
\end{rem}

\begin{thm}\label{cyclic_ext}
Assume that for some $(n,m,m_0,\Pi,\L)$ we have
\begin{itemize}
\item[($a$)] $nm_0$ divides $[K_{m,m_0}(p^\frac{1}{m}):K_{m,m_0}]$;
\item[($b$)] if $m$ is even, then 
$[K_{m,m_0}(\zeta_{2m},p^\frac{1}{m}):K_{m,m_0}(p^\frac{1}{m})]=2$;
\item[($c$)] if $\Pi\cap\RR\neq\emptyset$, then $nm_0$ is even.
\end{itemize}
Then there exist  infinitely many $\Omega$-admissible primes, where
$\Omega=(n,m,m_0,\Pi,\L,\eps)$ for some $\eps\in\{\pm 1\}$.
\end{thm}
\begin{proof}%[Proof of theorem~\ref{cyclic_ext}]
Suppose first that $m$ is even. 
%The group $G=\Gal(K_{m,m_0}(\zeta_{2m},p^{1/m})/K_{m,m_0})$ is the Cartesian product of the cyclic groups $\Gal(K_{m,m_0}(\zeta_{2m})/K_{m,m_0})$, and
%$\Gal(K_{m,m_0}(p^{1/m})/K_{m,m_0})$ over the Galois group \[\Gal((K_{m,m_0}(\zeta_{2m})\cap K_{m,m_0}(p^{1/m}))/K_{m,m_0}),\]
It follows from $(b)$ that 
\[G=\Gal(K_{m,m_0}(\zeta_{2m},p^{1/m})/K_{m,m_0})\cong\Gal(K_{m,m_0}(p^{1/m})/K_{m,m_0})\oplus \ZZ/2\ZZ.\]
Thus, there exists $h\in G$ that projects to a generator of 
$\Gal(K_{m,m_0}(\zeta_{2m})/K_{m,m_0})$, and to an element of order $nm_0$ in
$\Gal(K_{m,m_0}(p^{1/m})/K_{m,m_0})$.

It follows from the Chebotarev density theorem that there exist infinitely many primes $s$ such that $s$ is split in $K_{m,m_0}$, and $h$ is equal to the Frobenius automorphism $F_s$ of the local field extension $\Hat{K}_{m,m_0}(\zeta_{2m},p^{1/m})/\QQ_s$, where
$\Hat{K}_{m,m_0}\cong\QQ_s$ is the completion at a prime ideal over $s$.
Let $s$ be such a prime. Since $s$ splits in $\QQ(\zeta_{m})$, the number $m$ divides $s-1$. 

Let $\p$ be an ideal of $K_{m,m_0}(\zeta_{2m},p^{1/m})$ over $s$.
Let $\Delta_-$ be the subgroup of $K(\zeta_m)$ generated by $\Delta_{m_0}$, $p$, and $-1$.
According to Kummer theory, the homomorphism 
\[{\Delta_-(\QQ_s^*)^m}/(\QQ_s^*)^m\to\mu_m \]
corresponding to $h$ is a composition
\[{\Delta_-(K(\zeta_m)^*)^m}/(K(\zeta_m)^*)^m\to\Delta_-(\QQ_s^*)^m/(\QQ_s^*)^m\to\mu_m \]
of the natural monomorphism and a homomorphism given by $x\mapsto F_s(x^{1/m})/x^{1/m}$.
It follows that
\[\ell^{\frac{s-1}{m}}\equiv F_s(\ell^{\frac{1}{m}})/\ell^{\frac{1}{m}}\equiv 1\pmod\p,\]
and $\ell^{\frac{s-1}{m}}\equiv 1\pmod s$.

The order of $h$ restricted to $\Gal(K_{m,m_0}(p^{1/m})/K_{m,m_0})$ is equal to $nm_0$.
Hence, the order of
\[p^{\frac{s-1}{m}}\equiv F_s(p^{1/m})/p^{1/m}\pmod\p\] in $\FF_s^*$ is equal to $nm_0$.
According to Lemma~\ref{Kummer_lemma},
\[[\Hat K_{m,m_0}(p^{1/m}):\Hat K_{m,m_0}]=[\QQ_s(p^{1/m}):\QQ_s]=nm_0;\]
therefore,  $q^{m_0/m}=p^{nm_0/m}\in\QQ_s$.

The restriction of  $F_s$ to $\Gal(\Hat K_{m,m_0}(\zeta_{2m}))$ generates 
$\Gal(\QQ_s(\zeta_{2m})/\QQ_s)$, hence, the order of 
\[\frac{F_s(\zeta_{2m})}{\zeta_{2m}}=\zeta_{2m}^{s-1}=(-1)^{\frac{s-1}{m}}\]
is equal to $2$; therefore, $\frac{s-1}{m}$ is odd.
It follows that $-1$ is \emph{not} an $m$-th power in $\QQ_s$, but $-1=z_s^{m/2}$ for some non-square $z_s\in\QQ_s$.

Since $K_{m,m_0}$ is unramified at $s$, it follows from the Hensel lemma that $\pi^{m_0}$ is a $m$-th power in $\QQ_s(\pi)$ if and only if $\pi$ is an $m$-th power modulo $\p$.
Let $\pi,\pi'\in\Pi$, and let $\delta=(\pi/\pi')^{m_0}$. We have
\[\delta^{\frac{s-1}{m}}\equiv F_s(\delta^{\frac{1}{m}})/\delta^{\frac{1}{m}}\equiv 1\pmod\p;\]
therefore, $\delta$ is an $m$-th power in $\QQ_s$.

Note that $\QQ_s(\pi)=\QQ_s(\pi')$, and $\eps\pi^{m_0}$ is a $m$-th power in $\QQ_s(\pi)$ if and only if $\eps(\pi')^{m_0}$ is a $m$-th power in $\QQ_s(\pi')$.
%$\eps\in\{\pm 1\}$
If $\pi'=q/\pi$, then $\pi^{2m_0}/q^{m_0}$ is an $m$-th power in $\QQ_s$.
Since $q^{m_0/m}\in\QQ_s$, we proved that $\pi^{m_0}=x_\pi^{m/2}$ for some $x_\pi\in\QQ_s$. 
If $\pi^{m_0}$ is an $m$-th power in $\QQ_s$, then 
$s$ is $\Omega$-admissible, where $\Omega=(n,m,m_0,\Pi,\L,1)$.
If $\pi^{m_0}$ is not an $m$-th power in $\QQ_s$, then
$x_\pi$ is not a square in $\QQ_s$; therefore, $x_\pi z_s=\omega_\pi^2$ for some $\omega_\pi\in\QQ_s$.
We finally obtain that \[-\pi^{m_0}=z_s^{m/2}x_\pi^{m/2}=\omega_\pi^m\]
is an $m$-th power in $\QQ_s$, and $s$ is $\Omega$-admissible,
where $\Omega=(n,m,m_0,\Pi,\L,-1)$.

Suppose that $m$ is odd. 
There exists $h\in\Gal(K_{m,m_0}(p^{1/m})/K_{m,m_0})$ of order $nm_0$.
According to the Chebotarev density theorem, there exist infinitely many primes $s$ such that $s$ is split in $K_{m,m_0}$, and $h$ is equal to the Frobenius automorphism $F_s$ of the local field extension $\Hat{K}_{m,m_0}(p^{1/m})/\QQ_s$.
As before, if $s$ is such a prime, $m$ divides $s-1$, and 
$\ell^{\frac{s-1}{m}}\equiv 1\pmod s$.
Finally, the same argument shows that $nm_0$ is equal to the order of $p^\frac{s-1}{m}$ in $\FF_s^*$, and for all $\pi\in\Pi$ we have $\pi^{2m_0}=x_\pi^m$ for some $x_\pi\in\QQ_s$.
Since $m=2a+1$ is odd, $\pi^{m_0}=\omega_\pi^m$, where $\omega_\pi=\pi^{m_0}/x_\pi^a\in\QQ_s$.
It follows that $s$ is $\Omega$-admissible.
\end{proof}

\begin{corollary}\label{adm_prime_cor}
Fix $(n,\Pi,\L)$. Then there exist $\eps\in\{\pm 1\}$ and natural numbers $m_0$ and $n_0$ such that there exist infinitely many $\Omega$-admissible primes, 
where $\Omega=(n,nn_0m_0,m_0,\Pi,\L,\eps)$. 
\end{corollary}
\begin{proof}
Choose even $m_0$ such that $\Delta_{m_0}$ is torsion-free. 
%Note that if $\Pi\cap\RR\neq\emptyset$, then $nm_0$ is even. 
%Indeed, if $n$ is odd, then $\pm\sqrt{q}\in\Pi$, and $m_0$ is even.
Let $N_0$ be the number of Proposition~\ref{M_prop}. 
If $m$ is a natural number such that $nm_0N_0$ divides $m$, then, according to Proposition~\ref{M_prop}, the conditions of Theorem~\ref{cyclic_ext} hold. 
\end{proof}

\subsection{$\Pi$-admissible primes.}
Assume that $\Pi\cap\RR=\emptyset$.
We show that there are infinitely many $\Omega$-admissible primes, 
where $\Omega=(n,m,1,\Pi,\L,1)$, and $m$ divides $2n$.
We call such prime numbers \emph{$\Pi$-admissible}.
The following technical lemma is a variant of Schinzel’s theorem on abelian binomials~\cite{Sch}. It is a special case of~\cite[ Lemma 3.2]{IM}. 
We give a simple proof for the sake of completeness.

\begin{lemma}\label{sqrt_lemma}
Let $E^+/\QQ$ be a totally real Galois extension.
If $z\in E^+$ such that $z^a\in\QQ^+$, then $z^2\in\QQ$.
\end{lemma}
\begin{proof}
The Galois envelope $E_z\subset E^+$ of $\QQ(z)$ is totally real. 
If $z^2\not\in\QQ$, then there exists $b>2$ such that $\zeta_bz\in E_z$;
therefore, $E_z$ is not totally real. A contradiction.
\end{proof}

\begin{deff}
Fix $(n,m_0,\Pi,\L,\eps)$. Let $\DB_{m_0}\subset\overline{\QQ}^*$ be the group generated by $\L$ and all the products $\eps^{w(\pi)}\pi^{m_0}$, where $\pi\in\Pi$.
Put $K_\Pi=\QQ(\DB_{m_0})$.
\end{deff}

Fix $\Omega=(n,m,m_0,\Pi,\L,\eps)$.
Put $K_{\Pi,m,m_0}=K_\Pi(\zeta_m,(\DB_{m_0})^{1/m}).$
Let $\DB_\eps\subset\overline{\QQ}^*$ be the group generated by the products $\eps^{w(\pi)}\pi^{m_0}$.    

\begin{prop}\label{CM_prop}
Let $m=nn_0m_0$. Assume that $\Pi\cap\RR=\emptyset$.
\begin{enumerate}
    \item  Then $K_m=K_\Pi(\zeta_{2m},\DB_\eps^{1/2nm_0})$ is a $CM$-field.
    \item  If $m$ is odd or $n_0$ is even, then $nm_0$ divides 
    $[K_{\Pi,m,m_0}(p^\frac{1}{m}):K_{\Pi,m,m_0}]$.
    %\item  If $n_0$ is even, then $[K_{m,m_0}(p^\frac{1}{m}):K_{m,m_0}]=m/d$, where $d$ divides $n_0$.
\end{enumerate}
\end{prop}
\begin{proof}
      Part $(1)$ is standard. There is a function $w:\DB_\eps\to\ZZ$ such that 
      if $\delta\in\DB_\eps$, then $\delta\Bar\delta=q^{m_0w(\delta)}$. Hence,
    \[\delta^{1/2nm_0}\Bar\delta^{1/2nm_0}=\pm p^{w(\delta)/2}\] is real, and
     the totally complex field $\QQ(\delta^{1/2nm_0})$ is a quadratic extension of the totally real field generated by $\delta^{1/2nm_0}+\Bar\delta^{1/2nm_0}$ and $\pm p^{w(\delta)/2}$. Finally, the compositum of $CM$-fields is a $CM$-field.
    
We now prove $(2)$.
 According to Lemma~\ref{Kummer_lemma2}, if 
 $[K_{\Pi,m,m_0}(p^\frac{1}{m}):K_{\Pi,m,m_0}]=m/d,$ then
\[p^{m/d}\in \DB_{m_0}(K_\Pi(\zeta_m)^*)^m.\] 
We can write $p^{m/d}=\delta\lambda x^m$, where $\delta\in\DB_\eps$, and $\lambda$ is a product of primes from $\L$. Hence, 
\[p^{n_0/d}\lambda^{-n_0/m}\in K_m^+,\]
where the totally real subfield $K_m^+$ of $K_m$ is a Galois extension of $\QQ$. 
By Part $(1)$ and Lemma~\ref{sqrt_lemma}, $d$ divides $2n_0$.
If $d$ is odd, then $d$ divides $n_0$.

Assume that $n_0$ is even. Then 
\[p^{n_0/2d}\lambda^{-n_0/2m}=\delta^{1/2nm_0}x^{n_0/2}\in K_m^+.\]
As before, $d$ divides $n_0$.
The proposition is proved.
\end{proof}

\begin{thm}\label{cyclic_ext3}
Assume that $\Pi\cap\RR=\emptyset$.
Assume that for some $\Omega=(n,m,m_0,\Pi,\L,\eps)$ the number
$nm_0$ divides \[[K_{\Pi,m,m_0}(p^\frac{1}{m}):K_{\Pi,m,m_0}].\]
Then there exist infinitely many $\Omega$-admissible primes.
\end{thm}
\begin{proof}
The proof is similar to the proof of Theorem~\ref{cyclic_ext}. There exists 
\[h\in\Gal(K_{\Pi,m,m_0}(p^{1/m})/K_{\Pi,m,m_0})\] of order $nm_0$.
According to the Chebotarev Density Theorem, there exist infinitely many primes $s$ such that $s$ splits completely in $K_{\Pi,m,m_0}$, and $h$ is equal to the Frobenius automorphism $F_s$ of the local field extension $\Hat{K}_{\Pi,m,m_0}(p^{1/m})/\QQ_s$.

Let $\DB_+$ be the subgroup of $K_\Pi(\zeta_m)$ generated by $\DB_{m_0}$ and $p$.
The homomorphism 
\[{\DB_+(K_\Pi(\zeta_m)^*)^m}/(K_\Pi(\zeta_m)^*)^m\to\mu_m \]
given by $h$ is induced by $x\mapsto F_s(x^{1/m})/x^{1/m}$. 
As before, if $s$ is such a prime, then $m$ divides $s-1$, 
\[\ell^{\frac{s-1}{m}}\equiv 1\pmod s,\] and $nm_0$ is equal to the order of $p^\frac{s-1}{m}$ in $\FF_s^*$

Let $\delta=\eps\pi^{m_0}$. We have
\[\delta^{\frac{s-1}{m}}\equiv F_s(\delta^{\frac{1}{m}})/\delta^{\frac{1}{m}}\equiv 1\pmod\p;\]
therefore, $\delta$ is an $m$-th power modulo $\p$.
Since $K_{\Pi,m,m_0}$ is unramified at $s$, it follows from the Hensel lemma that $\delta$ 
is an $m$-th power in $\QQ_s$, and $s$ is $\Omega$-admissible.
\end{proof}

\begin{thm}\label{n_adm}
Fix $(n,\Pi,\L)$ such that $\Pi\cap\RR=\emptyset$.
\begin{enumerate}
\item Then there exist infinitely many $\Omega$-admissible primes, where
 $\Omega=(n,2n,1,\Pi,\L,1)$.
\item If $n$ is odd, then there exist infinitely many $\Omega$-admissible primes, where
 $\Omega=(n,n,1,\Pi,\L,1)$.
\item Assume that $n$ is even, and \[ \sqrt{q}\not\in\DB_{1}(K_\Pi(\zeta_n)^*)^n. \eqno{(*)} \]
   Then there exist infinitely many $\Omega$-admissible primes, where
 $\Omega=(n,n,1,\Pi,\L,1)$.
\end{enumerate}
\end{thm}
\begin{proof}
If $m=n$ is odd, or $n_0=2$ then, by Proposition~\ref{CM_prop}, the conditions of Theorem~\ref{cyclic_ext3} are satisfied for $\Omega=(n,nn_0,1,\Pi,\L,1)$.
This proves $(1)$ and $(2)$.
According to Lemma~\ref{Kummer_lemma2}, if
\[[K_{\Pi,n,1}(p^\frac{1}{n}):K_{\Pi,n,1}]=n/d\] then 
 $p^{n/d}=\delta\lambda x^n$, where $\delta\in\DB_\eps$ and $\lambda$ is a product of primes from $\L$. Hence, 
%\[p^{2n/d}=(\delta\overline{\delta}) \lambda^2(x\overline{x})^n=q^{a}\lambda^2(x\overline{x})^n\]
%for some $a\in\ZZ$. Therefore, 
\[p^{1/d}\lambda^{-1/n}\in K_n^+.\]
According to Proposition~\ref{CM_prop}.(1) and Lemma~\ref{sqrt_lemma}, $d$ divides $2$.
We see that \[[K_{\Pi,n,1}(p^\frac{1}{n}):K_{\Pi,n,1}]=n\] if and only if $(*)$ holds.
 We apply Theorem~\ref{cyclic_ext3} again.
\end{proof}

\subsection{Admissible field extensions.}
Let $\Pi$ be a $\Gal(\BQQ/\QQ)$-invariant set of Weil numbers over $\FF_q$.
Fix $\Omega=(n,m,m_0,\Pi,\L,\eps)$. 
%Assume that if $\Pi\cap\RR\neq\emptyset$, then $\eps=-1$.

\begin{thm}\label{inv_thm}
Let $L/\QQ$ be an $\Omega$-admissible extension. 
For any $\pi\in\Pi$ let \[\HH_{\pi}=(L\otimes\QQ(\pi)/\QQ(\pi),g,\eps\pi^{m_0})\]
be the cyclic algebra over $\QQ(\pi)$. If $d_\pi=m/e_\pi$, then
     \[\HH_{\pi}\cong \Mat_{d_\pi}(\End^\circ(B_\pi)).\]
\end{thm}%
\begin{proof}
First, we show that for any $\ell\neq p$ the invariants of $\HH_{\pi}$ are trivial.
Let $\lambda=\eps\pi^{m_0}$.
If $\ell\neq s$, then $\ell$ is unramified in $L$. 
Since $\lambda$ is an $\ell$-unit, it is a local norm~\cite[Corollary on p. 29]{CF}; 
according to Theorem~\ref{Th30_4}, invariants of $\HH_{\pi}$ are trivial at all the primes over $\ell$. If $\ell=s$, then $\lambda$ is locally an $m$-th power; 
in particular, it is a local norm and $\HH_{\pi}$ is trivial over $s$ again.

Second, for any prime ideal $v$ of $\QQ(\pi)$ over $p$ the invariant of $\HH_{\pi}$ at $v$ is equal to  $\frac{v(\pi)}{v(q)}[\QQ(\pi)_v:\QQ_p].$
Indeed, let $f_v$ be the inertia index of $\QQ(\pi)_v$ over $\QQ_p$, and let $g_v=gcd(nm_0,f_v)$.
Then \[L\otimes_\QQ\QQ(\pi)_v\cong (L(\pi)_v)^{n_0g_v},\] and, according to Theorem~\ref{Th30_4},
$\HH_\pi\otimes_{\QQ(\pi)}\QQ(\pi)_v$ is the matrix algebra over the cyclic algebra 
\[(L(\pi)_v/L_v,g,\eps\pi^{m_0})\cong(L(\pi)_v/L_v,g^{f_v/g_v},(\eps\pi^{m_0})^{f_v/g_v}),\]
where $g^{f_v/g_v}$ is the Frobenius automorphism for the extension $L(\pi)_v/L_v$.
According to Proposition~\ref{Hasse_inv}, the invariant of this algebra
is equal to \[\frac{v(\pi^{m_0f_v/g_v})}{nm_0/g_v}=\frac{v(\pi)f_v}{n}=\frac{v(\pi)}{v(q)}[\QQ(\pi)_v:\QQ_p].\] 

Finally, we compare invariants at infinity.
If $\pi$ is not real, then the invariant of $\HH_{\pi}$ at infinity is trivial.
If $\pi\in\Pi$ is real, then the invariant of $\HH_{\pi}$ at infinity is non-trivial by part $(7)$ of Definition~\ref{def_omega}. 
We proved that $\HH_{\pi}$ and $\End^\circ(B_\pi)$ are Brauer equivalent. The theorem now follows from dimension count.
\end{proof}

\begin{corollary}
Assume that $L\subset\QQ(\zeta_s)$, where the prime $s$ is unramified in $\QQ(\Pi)$. 
An extension $L/\QQ$ is $\Pi$-admissible if and only if $L/\QQ$ is $\Omega$-admissible, where
$\Omega=(n,n,1,\Pi,\L_\Pi,1)$.    
An extension $L/\QQ$ is weakly $\Pi$-admissible if $L/\QQ$ is $\Omega$-admissible, where 
    $\Omega=(n,2n,1,\Pi,\L_\Pi,1)$.
\end{corollary}
\begin{proof}
 The direction ``if'' is clear. To prove ``only if'' we only have to show that if 
    $\eps\pi^{m_0}$ is a local norm from $L(\pi)$ to $\QQ(\pi)$, then it is an $m$-th power in $\QQ_s(\pi)$. If $s$ is unramified in $\QQ(\Pi)$, then it can be proved as in Remark~\ref{local_norm_rem}.
\end{proof}

\begin{prop}\label{adm_ext}
If there exists an $\Omega$-admissible prime $s$, then there exists an $\Omega$-admissible extension $L/\QQ$. 
\end{prop}
\begin{proof}
Let $L\subset\QQ(\zeta_s)$ be the subfield of degree $m$ over $\QQ$. We will prove that $L/\QQ$ is $\Omega$-admissible.

Clearly, $L$ is unramified at $p$ and tamely ramified at $s$.
According to Lemma~\ref{Kummer_lemma}, $L/\QQ$ splits completely at any $\ell\in\L$, and $[\Hat{L}_p:\QQ_p]=nm_0$. There is nothing to prove in $(5)$ and $(6)$.

Assume that $\Pi\cap\RR\neq\emptyset$. Since the number $(s-1)/m$ is odd, $L$ is a CM-field.
By the definition of an $\Omega$-admissible prime, the order of $p^\frac{s-1}{m}$ is equal to $nm_0$ in $\FF_s^*$, and $nm_0$ is even. We claim that $p^{nm_0/2}=q^{m_0/2}$ is not an $m$-th power in $\QQ_s$. 
Indeed, assume that $p^{nm_0/2}=x^m$, where $x\in\QQ_s$. Then 
\[(p^\frac{s-1}{m})^{nm_0/2}\equiv(x^m)^\frac{s-1}{m}\equiv x^{s-1}\equiv 1\pmod{s}.\]
A contradiction.

Let $\pi\in\Pi\cap \RR$. Since $\eps\pi^{m_0}$ is an $m$-th power and $q^{m_0/2}>0$ is not an $m$-th power, it follows that \[\eps\pi^{m_0}=-q^{m_0/2}<0.\]
We showed that $L/\QQ$ is $\Omega$-admissible.
\end{proof}

We now combine Theorem~\ref{n_adm} and Proposition~\ref{adm_ext}.

\begin{thm}\label{Pi_adm_ext2}
Fix $(n,\Pi)$. Assume that $\Pi\cap\RR=\emptyset$.
\begin{enumerate}
\item If $n$ is odd, then there exist infinitely many $\Pi$-admissible extensions.
\item If $n$ is even, then there exist infinitely many weakly $\Pi$-admissible extensions.
\item Assume that $n$ is even, and \[ \sqrt{q}\not\in\DB_{1}(K_\Pi(\zeta_n)^*)^n. \eqno{(*)} \]
   Then there exist infinitely many $\Pi$-admissible extensions.
\end{enumerate}
\end{thm}

In some situations we can prove that $\Pi$-admissible extensions do not exist.

\begin{prop}\label{no_adm}
 Let $\Omega=(n,nm_0,m_0,\Pi,\L,\eps)$.
 Assume that $n$ is even. If 
 \[q^{m_0/2}\in\Delta_{m_0}(K(\zeta_m)^*)^m,\] 
 then there are no $\Omega$-admissible extensions $L/\QQ$ such that the ramified prime of $L/\QQ$ is totally split in $K$.
\end{prop}
\begin{proof}
Let $L/\QQ$ be such an $\Omega$-admissible extension. 
Since $L/\QQ$ is tamely ramified at $s$ and cyclic, according to the Kronecker--Weber Theorem, $L\subset \QQ(\zeta_s)$. By definition, $nm_0$ divides the local degree $[\Hat{L}_p:\QQ_p]$.

By assumption, the extension $K/\QQ$ is totally split at $s$, and any $\delta\in\Delta_{m_0}$ is an $m$-th power modulo $s$. In particular, $q^{m_0/2}=p^{nm_0/2}$ is a $m$-th power in $\FF_s^*$, that is, the order of $p^{\frac{s-1}{m}}$ divides $nm_0/2$. According to Lemma~\ref{Kummer_lemma}, $[\hat{L}_p:\QQ_p]$ divides $nm_0/2$. A contradiction.
\end{proof}

\begin{ex}\label{ex_main}
Let $K=\QQ(\zeta_4)$. Assume that $p\equiv 1\pmod 4$. Then $p=\pi_1\Bar\pi_1$ is a product of two $p$-Weil numbers. Put $\pi=p\pi_1^2$, and $\Pi=\{\pi,q/\pi\}$.
Then $n=4$, and \[\sqrt{q}=p^2=\pi^2\cdot\pi_1^{-4}\in\Delta_1(K^*)^4.\]
If $L$ is a $\Pi$-admissible extension of degree $4$, then $L\subset\QQ(\zeta_s)$;
therefore, $s\equiv 1\pmod{4}$ splits in $K$. 
According to Proposition~\ref{no_adm}, there are no $\Pi$-admissible extensions.    
\end{ex}

\begin{rem}
 The conditions of Theorem~\ref{Pi_adm_ext2}.(3) are not necessary. 
 In~\cite{Examples} we prove that there exist admissible field extensions such that 
  the ramified prime of $L/\QQ$ is not totally split in $K$, but the condition $(*)$ 
  of Theorem~\ref{Pi_adm_ext2}.(3) is not satisfied.
\end{rem}

\subsection{Generalized and twisted Deligne modules}\label{sec_tw_Del}
Let $\Pi$ be a finite Galois invariant set of Weil numbers over $\FF_q$, and
let $S\subset E_\Pi$ be a finite extension of $R_\Pi$.
Let $\L_S$ be the finite set of primes $\ell\neq p$, where $S$ is not maximal. 
According to Theorem~\ref{Omega_main_thm} there exists 
an $\Omega$-admissible cyclic extension $L/\QQ$ ramified at a single prime $s\not\in\L_S$, where $\Omega=(n,m,m_0,\Pi,\L_S,\eps)$. 
Let \[\HH_{\Pi,L}=\HH_{\Pi,L}(\eps F_\Pi^{m_0}).\]
According to Theorem~\ref{inv_thm}, if $A_{L,\Pi}=\prod_{\pi\in\Pi_0} B_\pi^{m/e_\pi}$, then
$\HH_{\Pi,L}\cong\End^\circ(A_{\Pi,L}).$

\begin{notation}\label{twisting_not}
Since $[\Hat{L}_p:\QQ_p]=nm_0$ there is an isomorphism
\[\OO_L\otimes\ZZ_p\cong W(\FF_{q^{m_0}})^{n_0}.\]
Let $\rho_p\in\OO_L\otimes\ZZ_p$ be the element corresponding to
\[(1,\dots, 1, p)\in W(\FF_{q^{m_0}})^{n_0}\] under this isomorphism.
Choose $\rho\in\OO_L$ such that $\rho\otimes 1-\rho_p\in pq^{m_0}\OO_L\otimes\ZZ_p$. 
\end{notation}

Clearly, $N_L(\rho_p)=q^{m_0}$, and $N_{L/\QQ}(\rho)\equiv q^{m_0}\pmod{pq^{m_0}}$.   

\begin{deff}
Let $\T=\T_{S,L}(\eps F_\Pi^{m_0},\rho)\subset\HH_{\Pi,L}$ be the order of Example~\ref{L_order}. 
The category of \emph{twisted Deligne $S$-modules over $\OO_L$} is the category
\[\TDel_{S,L}=\TF{\T^\op}.\]  
%The category of \emph{rigid twisted Deligne $S$-modules over $\OO_L$} is the category:
%\[\RTDel_{S,L}=\RM{\T}.\]
Let $H$ be the endomorphism of Definition~\ref{H_operator}, and let 
$\T^\max$ be the order generated by $\T$ and $H$.
The category of \emph{twisted Deligne modules invariant under $H$} is the category  
\[\TDel_{S,L,H}=\TF{\T^{\max,\op}}.\]
\end{deff}

If we put $S=R_\Pi$ in this definition, we get the category of twisted Deligne modules over $\OO_L$ with Weil support in $\Pi$:
\[\TDel_{\Pi,L}=\TF{\T_{R_\Pi,L}(\eps F_\Pi^{m_0},\rho)^\op}.\]
According to Proposition~\ref{balanced_order}, there are tame structures over $\T$ and $\T^\max$ that are $\ell$-balanced for all primes $\ell\not\in\{s,p\}$, and according to Corollary~\ref{s_max_order}, the tame structure over $\T^\max$ is $s$-balanced. 

Note that $W(\FF_{q^{m_0}})\locp$ is unramified over $\QQ_p$; therefore,
$\eps\in\{\pm 1\}$ is a norm from $W(\FF_{q^{m_0}})$. To save ink we write: $G=G_\lambda$ and $\CG=\CG_\lambda$, where $\lambda=\eps F_\Pi^{m_0}$. 

\begin{deff}
Put $n_0=m/(nm_0)$, and choose $\mu\in W(\FF_{q^{m_0}})$ such that 
\[\eps=N_{W(\FF_{q^{m_0}})\locp/\QQ_p}(\mu).\]
Define a structure of a $\T^\op$-module on the contravariant \Die module (that is, right $D_S^\op$-module)
    \[M_{S,L}=(W(\FF_{q^{m_0}})\otimes_{W}D_S)^{n_0}=\oplus_{i=1}^{n_0} W(\FF_{q^{m_0}})\otimes_{W}D_S\] 
    as follows. There is a tautological action of $\OO_L\otimes\ZZ_p\cong W(\FF_{q^{m_0}})^{n_0}$.
    Define the action of $G$ and $\CG$ by the following formulas:
    \[G(\oplus_{i=1}^{n_0}(x_i\otimes y_i))=(\mu\sigma x_{n_0}\otimes Fy_{n_0}) \oplus
    (\oplus_{i=1}^{n_0-1}(x_{i}\otimes y_{i})),\text{ and }\]
%\[\CG(\oplus_{i=1}^{n_0}(x_i\otimes y_i))=\oplus_{i=1}^{n_0}(\mu\sigma x_i\otimes Vy_i).\] 
   \[\CG(\oplus_{i=1}^{n_0}(x_i\otimes y_i))=
   (\oplus_{i=2}^{n_0}(x_{i}\otimes y_{i}))\oplus (\mu^{-1}\sigma^{-1} x_1\otimes Vy_1).\]
\end{deff}

According to Lemma~\ref{quasi_free}, the module $M_{S,L}$ is quasi-free,
and by Proposition~\ref{wide_AV}, there exists a wide abelian variety $A_\max$ over $\T^\max$
such that $M^*(A_\max)\cong M_{S,L}$.
Hence, there are functors $\cT_\ell, \hT_\ell$ and $D_\ell$ associated with the tame structure 
$\{T_\ell^*(A_\max)\}_{\ell\neq p}$.
According to Proposition~\ref{Tate_eq}, any object of $\TDel_{S,L,H}$ is $\ell$-rigid for all $\ell\neq p$.

\begin{prop}\label{m0prop2}
    If $m_0=1$, then $\T$ is $p$-balanced, and
    any object of $\TDel_{S,L}$ is $p$-rigid.
\end{prop}
\begin{proof}
By definition of $\rho$, if $\CG_p=\rho_pG^{-1}$, then $\CG_p\in \T_p$, and 
$\CG_p^m=\eps V_\Pi^{m_0}$. There is a homomorphism $D_S^\op\to \T_p^\op$ given by 
$F\mapsto \mu^{-1}G^{n_0}$ and $V\mapsto \mu\CG_p^{n_0}$.

  Let $v_1,\dots, v_{n_0}$ be a basis of $\OO_L\otimes\ZZ_p\cong W^{n_0}$ over $W$. 
  Consider the action of $\T_p^\op$ on $M_{S,L}\cong D_S^{n_0}$. The elements $v_i$ act as diagonal elementary matrices  $I_{ii}\in\Mat_{n_0}(D_S^\op)$.
    If $1\leq i\leq n_0-1$, then 
    \[v_i G\in W^* I_{i(i-1)},\text{ and }\; v_i\CG_p\in W^*I_{(i-1)i}.\] 
    These matrices generate $\Mat_{n_0}(D_S^\op)\cong \End_{D_S}(M_{S,L})$ over $D_S^\op$, hence $\T$ is $p$-balanced.
    The proposition follows from Proposition~\ref{p_rigid_prop}.
\end{proof}

Now, from Proposition~\ref{wide_AV}, Theorem~\ref{main_thm_cyclic} and Corollary~\ref{main_thm} we obtain the main result of this section.

\begin{thm}\label{Lambda_thm}
Let $L/\QQ$ be an $\Omega$-admissible extension, where $\Omega=(n,m,m_0,\Pi,\L_S,\eps)$.
\begin{enumerate}
\item Let $\T^\max$ be the order generated by $\T_{S,L}(\eps F_\Pi^{m_0},\rho)$ and the operator $H$ of Definition~\ref{H_operator}. 
Then there exists a wide abelian variety $A_\max$ such that $\{T_\ell^*(A_\max)\}_{\ell\neq p}$ is a balanced tame structure. 
       \item The functor $\D'_{\T^\max}$ is an equivalence from 
    $\AV_{S}$ to the category  \[\TDel_{S,L,H}\locp\times_{\Md{D_{S}^\op\locp}}\TF{D_{S}^\op}.\]
        \item If $m_0=1$, then $\T^\max\cong\End(A_\max)$ and $\D_{\T^\max}$ 
        is an equivalence from $\AV_{S}$ to $\TDel_{S,L,H}$.
    \end{enumerate}
\end{thm}

Assume that there exists a $\Pi$-admissible extension $L/\QQ$ ramified at a prime $s$.
The category $\GDel_{\Pi,L}$ of generalized Deligne modules with Weil support in $\Pi$ is equivalent to the category of twisted Deligne modules $\TDel_{R_\Pi,L}$.
More generally, we define 
\[\GDel_{S,L}=\TDel_{S,L}\text{ and } \GDel_{S,L,H}=\TDel_{S,L,H}.\]

\begin{proof}[Proof of Theorem~\ref{main_thm_Del}.]
Let $\T=\T_{R_\Pi,L}(F_\Pi,p)$ be the order of Example~\ref{L_order}, and let 
$\T^\max$ be the order generated by $\T$ and the endomorphism $H$ of Definition~\ref{H_operator}. 
%According to Propositions~\ref{m0prop2} and~\ref{wide_AV}, 
According to Theorem~\ref{Lambda_thm}, there exists $A_\max$ such that 
\[\T\left[\frac{1}{s}\right]\cong \T^\max\left[\frac{1}{s}\right]
\cong\End(A_\max)\left[\frac{1}{s}\right],\]
and $D_\Pi=\D_{\T}[\frac{1}{s}]$ is an equivalence.

Let $\ell\notin \{s,p\}$ be a prime.
Choose an isomorphism $T_{\ell}^*(A_\max)\cong R_\Pi\otimes_\ZZ\OO_L\otimes_\ZZ\ZZ_\ell$.
According to Proposition~\ref{rigid_prop}.(1), there is an isomorphism
\[\D_\Pi(A)\otimes\ZZ_\ell\cong T_\ell(A)\otimes_{R_\Pi\otimes\ZZ_\ell}T_{\ell}^*(A_\max)\cong T_\ell(A)\otimes_\ZZ\OO_L.\]

According to Proposition~\ref{rigid_prop}.(2), 
\[\D_\Pi(A)\otimes_\ZZ\ZZ_p\cong M(A).\]
According to Remark~\ref{die_dual}, if we choose an isomorphism $W^*\cong W$, then
$M(A)\cong M_*(A)$. The theorem is proved.
\end{proof}

\begin{proof}[Proof of Theorem~\ref{main_thm_intro}.]
Immediate from Theorem~\ref{Pi_adm_ext2} and Theorem~\ref{Lambda_thm}.
\end{proof}


\begin{thebibliography}{99}

\bibitem[BKM]{BKM}J. Bergström, V. Karemaker, S. Marseglia. {\it Abelian varieties over finite fields with commutative endomorphism algebra: theory and algorithms.}
https://arxiv.org/abs/2409.08865

%\bibitem[CCO14]{LAV} C.-L. Chai, B. Conrad, F. Oort. {\it Complex multiplication and lifting problems.}  Mathematical  Surveys  and  Monographs,  vol.  195,  American Mathematical Society, Providence, RI, 2014.

\bibitem[CF67]{CF} {\it Algebraic number theory.}
Proceedings of an instructional conference organized by the London Mathematical Society
with the support of the International Mathematical Union.
Edited by J. W. S. Cassels and A. Fr\"ohlich.
Academic Press, London; Thompson Book Co., Inc., Washington, D.C. 1967


\bibitem[CS15]{CS15} T. Centeleghe and J. Stix.
{\it Categories of abelian varieties over finite fields I: Abelian varieties over $\FF_p$.} Algebra \& Number Theory 9 (2015), 225–265.
%DOI: 10.2140/ant.2015.9.225
 

\bibitem[CS23]{CS2} Centeleghe, T.G., Stix, J. {\it Categories of abelian varieties over finite fields II: Abelian varieties over $\FF_q$ and Morita equivalence.} Isr. J. Math. 257, 103–170 (2023). 
%https://doi.org/10.1007/s11856-023-2536-2

\bibitem[De69]{Del} P. Deligne, {\it Vari\'et\'es ab\'eliennes ordinaires sur un corps fini}, 
Invent.\ math.\ \textbf{8} (1969), 238--243.

%\bibitem[De78]{De} M. Demazure.
%{\it Lectures on $p$-divisible groups}.
% Lecture notes in mathematics 302, Springer, 1972.

\bibitem[Ho68]{Ho} T. Honda. {\it Isogeny classes of abelian varieties over finite fields.} J. Math. Soc. Japan 20(1--2): 83-95 (1968).


\bibitem[Ib82]{Ib} T. Ibukiyama. {\it On maximal orders of division quaternion algebras over the rational number field with certain optimal embeddings.}
Nagoya Math. J., 88:181--195, (1982).

\bibitem[IM98]{IM}I. M. Isaacs, D. P. Moulton. {\it Real fields and repeated radical extensions.}
Journal of Algebra, 201, (1998), 429--455. 

\bibitem[JKPRSBT]{EC18} B. W. Jordan, A. G. Keeton, B. Poonen, E. M. Rains, N. Shepherd-Barron, J. T. Tate. {\it Abelian varieties isogenous to a power of an elliptic curve.}
Compositio Mathematica, Volume 154, Issue 5, (2018), 934 -- 959.
%DOI: https://doi.org/10.1112/S0010437X17007990

\bibitem[KMRT]{KMRT} M.A. Knus, A. Merkurjev, M. Rost, J.-P. Tignol. {\it The Book of Involutions}. AMS Coll. Publications, Vol.44 (1998).

\bibitem[Mil06]{MK} J. Milne. {\it Motives over $\FF_p$.} 2006.
https://www.jmilne.org/math/articles/2006a.pdf

\bibitem[Mil08]{Milne} J. Milne. {\it Abelian varieties.} 2008.
 http://www.jmilne.org/math/CourseNotes/av.html

\bibitem[MO]{MO}  I. Reiner. {\it Maximal  orders}. London Mathematical Society Monographs, No.5. AcademicPress, London-New York, 1975, 395 pp.

\bibitem[Mum70]{Mum} D. Mumford. {\it Abelian varieties.}
Tata Institute of Fundamental Research Studies in Mathematics, No. 5.
Published for the Tata Institute of Fundamental Research, Bombay;
Oxford University Press, London 1970.

\bibitem[OSh20]{OO} A. Oswal, A. N. Shankar.
{\it Almost ordinary abelian varieties over finite fields.}
Journal of the London Mathematical Society,
Volume 101, Issue 3, 2020, pp. 923--937.
https://doi.org/10.1112/jlms.12291

\bibitem[Pink]{Pink} R. Pink. {\it Finite Group Schemes.} Lecture notes.  https://people.math.ethz.ch/~pink/FiniteGroupSchemes.html

\bibitem[PS19]{PS} A. Perucca, P. Sgobba. {\it Kummer theory for number fields and the reductions of algebraic numbers.} 
 International Journal of Number Theory, Vol. 15, No. 08, 2019, pp. 1617--1633. 

%\bibitem[Ry10]{Ry4} S. Rybakov.
%{\it The groups of points on abelian varieties over finite fields.} Cent. Eur. J. Math. 8(2), 2010, 282--288. arXiv:0903.0106v4

\bibitem[Ry]{Examples} S. Rybakov. {\it Polarizations on generalized Deligne modules.} To appear.

%\bibitem[Ry]{Examples} S. Rybakov. {\it Admissible field extensions and abelian varieties over finite fields.} To appear.

\bibitem[Sch77]{Sch} A. Schinzel. {\it Abelian binomials, power residues and exponential congruences.} Acta Arithmetica 32 (1977), 245-274.

\bibitem[Ta66]{Ta66} J. Tate. {\it Endomorphisms of abelian varieties over finite fields.} Inventiones mathematicae 1966, Volume 2, Issue 2, pp 134--144.

%\bibitem[Wa69]{Wa} W. Waterhouse. {\it Abelian varieties over finite fields.} Ann.\ scient.\ \'Ec.\ Norm.\ Sup., 4 serie {\bf 2}, 1969, 521--560.
%MR0314847

\bibitem[WM69]{WM} W. Waterhouse, J. Milne. {\it Abelian varieties over finite fields.}
Proc. Sympos. Pure Math., Vol. XX, State Univ. New York, Stony Brook, N.Y., 1969,  53--64.
Amer. Math. Soc., Providence, R.I., 1971.

\end{thebibliography}
\end{document}